%% file: arxiv_version.tex
\documentclass{article}

\usepackage[final]{neurips_2019}

\usepackage{transparent,url,examplep}
\usepackage[utf8]{inputenc} 
\usepackage[T1]{fontenc}    
\usepackage[hidelinks,colorlinks,linkcolor=blue,citecolor=blue]{hyperref}       
\usepackage{url}            
\usepackage{booktabs}       
\usepackage{nicefrac}       
\usepackage{microtype}      

\usepackage{amsmath,amsthm,amssymb,amsfonts}       
\usepackage{dsfont}
\usepackage{latexsym}
\usepackage{graphicx}
\usepackage{algorithmic}
\usepackage{hyperref}
\usepackage{cleveref}
\usepackage{autonum}
\usepackage{color}
\usepackage{mathrsfs}

\def\hat{\widehat}
\newcommand{\esp}{\mathbb{E}}

\newcommand{\re}{\mathbb{R}}

\DeclareMathOperator*{\tr}{trace}

\DeclareMathOperator*{\RE}{RE}
\DeclareMathOperator*{\TP}{TP_{\bfSigma}}
\DeclareMathOperator*{\IP}{IP_{\bfSigma}}
\DeclareMathOperator*{\ATP}{ATP_{\bfSigma}}
\DeclareMathOperator*{\TPI}{TP_{\bfI}}
\DeclareMathOperator*{\IPI}{IP_{\bfI}}
\DeclareMathOperator*{\ATPI}{ATP_{\bfI}}

\def\bfSigma{\boldsymbol{\Sigma}}
\def\bfDelta{\boldsymbol{\Delta}}
\def\bfA{\mathbf{A}}

\def\bfI{\mathbf{I}}
\def\bfM{\mathbf{M}}
\def\bfX{\mathbf{X}}

\def\bfZ{\mathbf{Z}}

\def\bX{\boldsymbol X}
\def\bY{\boldsymbol Y}

\def\bb{\boldsymbol b}

\def\by{\boldsymbol y}
\def\bu{\boldsymbol u}
\def\bv{\boldsymbol v}
\def\bw{\boldsymbol w}

\def\bbeta{\boldsymbol\beta}

\def\btheta{\boldsymbol\theta}
\def\bxi{\boldsymbol\xi}

\def\calN{\mathcal N}

\def\sa{\mathsf{a}}
\def\sb{\mathsf{b}}
\def\sc{\mathsf{c}}
\def\sd{\mathsf{d}}

\newcommand{\vertiii}[1]{{\left\vert\kern-0.4ex #1
\kern-0.4ex\right\vert}}

\DeclareMathOperator*{\argmin}{argmin}

\theoremstyle{plain}
\newtheorem{definition}{Definition}
\newtheorem{lemma}{Lemma}
\newtheorem{theorem}{Theorem}
\newtheorem{corollary}{Corollary}
\newtheorem{proposition}{Proposition}

\theoremstyle{definition}

\newtheorem{remark}{Remark}

\title{Outlier-robust estimation of a sparse linear model using $\ell_1$-penalized
Huber's $M$-estimator}

\author{
  Arnak S.~Dalalyan\\
  ENSAE Paristech-CREST\\
  \texttt{arnak.dalalyan@ensae.fr} \\
  \And
  Philip Thompson\\
  ENSAE Paristech-CREST\\
  \texttt{philipthomp@gmail.com}
}

\begin{document}

\maketitle

\begin{abstract}
    We study the problem of estimating a $p$-dimensional
    $s$-sparse vector in a linear model with Gaussian design and
    additive noise. In the case where the labels are contaminated
    by at most $o$ adversarial outliers, we prove that the
    $\ell_1$-penalized Huber's $M$-estimator based on $n$ samples
    attains the optimal rate of convergence $(s/n)^{1/2} + (o/n)$,
    up to a logarithmic factor. For more general design matrices, 
    our results highlight the importance of two properties: the 
    transfer principle and the incoherence property. These 
    properties with suitable constants are shown to yield the 
    optimal rates, up to log-factors, of robust estimation with 
    adversarial contamination. 
\end{abstract}

\section{Introduction}

Is it possible to attain optimal rates of estimation in
outlier-robust sparse regression using penalized empirical risk
minimization (PERM) with convex loss and convex penalties? Current
state of literature on robust estimation does not answer this
question. Furthermore, it contains some signals that might suggest
that the answer to this question is negative. First, it has been
shown in \citep[Theorem 1]{pmlr-v28-chen13h} that in the case of
adversarially corrupted samples, no method based on penalized
empirical loss minimization, with convex loss and convex penalty,
can lead to consistent support recovery. The authors then advocate
for robustifying the $\ell_1$-penalized least-squares estimators
by replacing usual scalar products by their trimmed counterparts.
Second,  \citep{chen:gao:ren2018} established that in the
multivariate Gaussian model subject to Huber's contamination,
coordinatewise median---which is the ERM for the $\ell_1$-loss---is
sub-optimal. Similar result was proved in
\citep[Prop.~2.1]{lai:rao:vempala2016} for the geometric median,
the ERM corresponding to the $\ell_2$-loss. These negative results
prompted researchers to use other techniques, often of higher
computational complexity, to solve the problem of outlier-corrupted
sparse linear regression.

In the present work, we prove that the $\ell_1$-penalized empirical
risk minimizer based on  Huber's loss is minimax-rate-optimal, up
to possible logarithmic factors. Naturally, this result is not valid
in the most general situation, but we demonstrate its validity
under the assumptions that the design matrix satisfies some
incoherence condition and only the response is subject to contamination.
The incoherence condition is shown to be satisfied by the Gaussian
design with a covariance matrix that has bounded and bounded away
from zero diagonal entries. This relatively simple setting is
chosen in order to convey the main message of this work:
\textit{for properly chosen convex loss and convex penalty functions,
the PERM is minimax-rate-optimal in sparse linear regression
with adversarially corrupted labels}.

To describe more precisely the aforementioned optimality result, let
$\mathcal D_n^\circ=\{(\bX_i, y^\circ_i); i=1,\ldots,n\}$ be iid
feature-label pairs such that $\bX_i\in\mathbb{R}^p$ are Gaussian
with zero mean and covariance matrix $\bfSigma$ and $y_i^\circ$
are defined by the linear model
\begin{align}
y^\circ_i = \bX_i^\top \bbeta^* +\xi_i,\qquad i=1,\ldots,n,
\label{lm:1}
\end{align}
where the random noise $\xi_i$, independent of $\bX_i$, is Gaussian with
zero mean and variance $\sigma^2$. Instead of observing the ``clean''
data $\mathcal D_n^\circ$, we have access to a contaminated version
of it, $\mathcal D_n=\{(\bX_i, y_i);i=1,\ldots,n\}$, in which a
small number $o\in\{1,\ldots,n\}$ of labels $y_i^\circ$ are replaced
by an arbitrary value. Setting $\theta_i^* = (y_i-y_i^\circ) /\sqrt{n}$,
and using the matrix-vector notation, the described model can be written as
\begin{align}
\bY = \bfX\bbeta^* +\sqrt{n}\,\btheta^*+\bxi,
\label{lm:2}
\end{align}
where $\bfX = [\bX_1^\top;\ldots;\bX_n^\top]$ is the
$n\times p$ design matrix, $\bY = (y_1,\ldots,y_n)^\top$ is the response
vector, $\btheta^* = (\theta_1^*,\ldots,\theta_n^* )^\top$ is the
contamination and $\bxi = (\xi_1,\ldots,\xi_n)^\top$ is the noise vector.
The goal is to estimate the vector $\bbeta^*\in \mathbb{R}^p$. The
dimension $p$ is assumed to be large, possibly larger than $n$ but,
for some small value $s\in\{1,\ldots,p\}$, the vector $\bbeta^*$ is
assumed to be $s$-sparse: $\|\bbeta^*\|_0 =
\text{Card}\{j:\beta^*\neq 0\}\le s$. In such a setting, it is
well-known that if we have access to the clean data $\mathcal{D}_n^\circ$
and measure the quality of an estimator $\hat\bbeta$ by the
Mahalanobis norm\footnote{In the sequel, we use notation $\|\bbeta\|_q
= (\sum_j |\beta_j|^q)^{1/q}$ for any vector $\bbeta\in\mathbb{R}^p$
and any $q\ge 1$.} $\|\bfSigma^{1/2}(\hat\bbeta-\bbeta^*)\|_2$,
the optimal rate is
\begin{align}\label{rate:1}
    r^\circ(n,p,s) = \sigma\Big(\frac{s\log (p/s)}{n}\Big)^{1/2}.
\end{align}
In the outlier-contaminated setting, \textit{i.e.}, when $\mathcal{D}_n^\circ$ is unavailable but one has access to $\mathcal{D}_n$, the minimax-optimal-rate \citep{chen2016} takes the form
\begin{align}\label{rate:2}
    r(n,p,s,o) = \sigma\Big(\frac{s\log (p/s)}{n}\Big)^{1/2} + \frac{\sigma o}{n}.
\end{align}
The first estimators proved to attain this rate
\citep{chen2016,Gao2017} were computationally intractable\footnote{In the sense that there is no algorithm  computing these estimators in time polynomial in $(n,p,s,o)$.} for large $p$, $s$ and $o$. This
motivated several authors to search for polynomial-time
algorithms attaining nearly optimal rate; the most relevant results will be reviewed later in this work.

The assumption that only a small number $o$ of labels are
contaminated by outliers implies that the vector $\btheta^*$ in \eqref{lm:2} is $o$-sparse. In order to take advantage of sparsity of both $\bbeta^*$ and $\btheta^*$ while ensuring computational tractability of the resulting estimator, a natural approach studied in several papers
\citep{Laska,nguyen:tran2013,dalalyan:chen2012} is to use some version of the $\ell_1$-penalized ERM. This corresponds to defining
\begin{align}\label{Lasso}
    \hat{\bbeta} \in \text{arg}\min_{\bbeta\in\mathbb{R}^p}
    \min_{\btheta\in\mathbb{R}^n}
    \Big\{\frac1{2n}\|\bY-\bfX^\top\bbeta-\sqrt{n}\,\btheta\|_2^2
    + \lambda_s\|\bbeta\|_1+\lambda_o\|\btheta\|_1\Big\},
\end{align}
where $\lambda_s,\lambda_o>0$ are tuning parameters. This estimator is very
attractive from a computational perspective, since it can be seen as the
Lasso for the augmented design matrix $\bfM = [\bfX, \sqrt{n}\,\bfI_n]$,
where $\bfI_n$ is the $n\times n$ identity matrix. To date, the best known
rate for this type of estimator is
\begin{align}
    \label{rate:3}
    \sigma\Big(\frac{s\log p}{n}\Big)^{1/2} + \sigma\Big(\frac{ o}{n}\Big)^{1/2},
\end{align}
obtained in \citep{nguyen:tran2013} under some restrictions on
$(n,p,s,o)$. A quick comparison of \eqref{rate:2} and \eqref{rate:3}
shows that the latter is sub-optimal. Indeed, the ratio of the two
rates may be as large as $(n/o)^{1/2}$. The main goal of the present
paper is to show that this sub-optimality is not an intrinsic
property of the estimator \eqref{Lasso}, but rather an artefact
of previous proof techniques. By using a refined argument, we
prove that $\hat\bbeta$ defined by \eqref{Lasso} does attain
the optimal rate under very mild assumptions.

In the sequel, we refer to $\hat\bbeta$ as $\ell_1$-penalized
Huber's $M$-estimator. The rationale for this term is that
the minimization with respect to $\btheta$ in \eqref{Lasso}
can be done explicitly. It yields \citep[Section 6]{Donoho2016}
\begin{align}\label{Huber:1}
    \hat\bbeta \in\text{arg}\min_{\bbeta\in\mathbb{R}^p}
    \Big\{ \lambda_o^2 \sum_{i=1}^n \Phi\Big(\frac{y_i-\bX_i^\top
		\bbeta}{\lambda_o\sqrt{n}}\Big) +
    \lambda_s \|\bbeta\|_1\Big\},
\end{align}
where $\Phi:\mathbb{R}\to \mathbb{R}$ is Huber's function defined
by $\Phi(u) = (\nicefrac12)u^2 \wedge (|u|-\nicefrac12)$.

To prove the rate-optimality of the estimator $\hat{\bbeta}$, we
first establish a risk bound for a general design matrix $\bfX$ not
necessarily formed by Gaussian vectors. This is done in the next section.
Then, in \Cref{sec:3}, we state and discuss the result showing that all
the necessary conditions are satisfied for the Gaussian design. Relevant
prior work is presented in \Cref{sec:3b},  while \Cref{sec:4} discusses
potential extensions. \Cref{sec:5} provides a summary of our results and
an outlook on future work. The proofs are deferred to the supplementary
material.

\section{Risk bound for the $\ell_1$-penalized Huber's $M$-estimator}
\label{sec:2}

This section is devoted to bringing forward sufficient
conditions on the design matrix that allow for
rate-optimal risk bounds for the estimator $\hat\bbeta$
defined by \eqref{Lasso} or, equivalently, by
\eqref{Huber:1}. There are two qualitative conditions
that can be easily seen to be necessary: we call them
restricted invertibility and incoherence. Indeed, even
when there is no contamination, \textit{i.e.}, the
number of outliers is known to be $o=0$, the matrix
$\bfX$ has to satisfy a restricted invertibility
condition (such as restricted isometry, restricted
eigenvalue or compatibility) in order that the Lasso
estimator \eqref{Lasso} does achieve the optimal
rate $\sigma \sqrt{(s/n)\log(p/s)}$.
On the other hand, in the case where $n=p$ and
$\bfX = \sqrt{n}\,\bfI_n$, even in the extremely
favorable situation where the noise $\bxi$ is zero,
the only identifiable vector is $\bbeta^*+\btheta^*$.
Therefore, it is impossible to consistently estimate
$\bbeta^*$ when the design matrix $\bfX$ is aligned
with the identity matrix $\bfI_n$ or close to be so.

The next definition formalizes what we call restricted
invertibility and incoherence by introducing three notions:
the transfer principle, the incoherence property and the
augmented transfer principle. We will show that these
notions play a key role in robust estimation by
$\ell_1$-penalized least squares.

\begin{definition}\label{def:TP:MI}
Let $\bfZ\in\re^{n\times p}$ be a (random) matrix and
$\bfSigma \in\re^{p\times p}$. We use notation $\bfZ^{(n)} =
\bfZ/\sqrt{n}$.
\begin{itemize}\itemsep=0pt
\item[\rm (i)] We say that $\bfZ$ satisfies the transfer
principle with $\sa_1\in(0,1)$ and $\sa_2\in(0,\infty)$,
denoted by $\TP(\sa_1;\sa_2)$, if for all $\bv\in\re^{p}$,
\begin{align}
\big\|\bfZ^{(n)}\bv\big\|_2\ge \sa_1\Vert\bfSigma^{1/2}
\bv\Vert_2-\sa_2\Vert\bv\Vert_1.\label{TP}
\end{align}
\item[\rm (ii)] We say that $\bfZ$ satisfies the incoherence property
$\IP(\sb_1;\sb_2;\sb_3)$ for some positive numbers $\sb_1$, $\sb_2$
and $\sb_3$, if for all $[\bv;\bu]\in\re^{p+n}$,
\begin{align}
|\bu^\top\bfZ^{(n)}\bv|\le \sb_1\big\|{\bfSigma}^{1/2}\bv\big\|_2
\Vert\bu\Vert_2+\sb_2\Vert\bv\Vert_1\Vert\bu\Vert_2+\sb_3\big\|
{\bfSigma}^{1/2}\bv\big\|_2\Vert\bu\Vert_1.\label{IP}
\end{align}
\item[\rm (iii)] We say that $\bfZ$ satisfies the augmented transfer
principle $\ATP(\sc_1;\sc_2;\sc_3)$ for some positive numbers
$\sc_1$, $\sc_2$ and $\sc_3$, if for all $[\bv;\bu]\in\re^{p+n}$,
\begin{align}
\|\bfZ^{(n)}\bv+\bu\|_2\ge \sc_1
\big\|[\bfSigma^{1/2}\bv;\bu]\big\|_2-\sc_2\|\bv\|_1
-\sc_3\|\bu\|_1.\label{ATP}
\end{align}
\end{itemize}
\end{definition}

These three properties are inter-related and related to extreme
singular values of the matrix $\bfZ^{(n)}$.
\begin{description}
	\item[(P1)] If $\bfZ$ satisfies $\ATP(\sc_1;\sc_2;\sc_3)$ then
	it also satisfies $\TP(\sc_1;\sc_2)$.
	\item[(P2)] If $\bfZ$ satisfies $\TP(\sa_1;\sa_2)$ and
	$\IP(\sb_1;\sb_2;\sb_3)$ then it also satisfies $\ATP(\sc_1;\sc_2;\sc_3)$
	with $\sc_1^2 = \sa_1^2-\sb_1-\alpha^2$, $\sc_2 = \sa_2+2\sb_2/\alpha$
	and $\sc_3 = 2\sb_3/\alpha$ for any positive $\alpha< \sqrt{\sa_1^2-\sb_1}$.
	\item[(P3)] If $\bfZ$ satisfies  $\IP(\sb_1;\sb_2;\sb_3)$, then it also
	satisfies $\IP(0;\sb_2;\sb_1+\sb_3)$
	\item[(P4)] Any matrix $\bfZ$ satisfies $\TPI(s_p(\bfZ^{(n)});0)$,
	and $\IPI(s_1(\bfZ^{(n)});0;0)$, where
	$s_p(\bfZ^{(n)})$ and $s_1(\bfZ^{(n)})$ are, respectively, the $p$-th
	largest and the largest singular values of $\bfZ^{(n)}$.
\end{description}
Claim (P1) is true, since if we choose $\bu=\mathbf 0$ in \eqref{ATP}
we obtain \eqref{TP}. Claim (P2) coincides with \Cref{lemma:aug:transfer:principle},
proved in the supplement. (P3) is a direct consequence of the inequality $\|\bu\|_2
\le\|\bu\|_1$, valid for any vector $\bu$. (P4) is a well-known characterization of
the smallest and the largest singular values of a matrix. We will show later on
that a Gaussian matrix satisfies with high probability all these conditions with
constants $\sa_1$ and $\sc_1$ independent of $(n,p)$ and $\sa_2$, $\sb_2$, $\sb_3$,
$\sc_2$, $\sc_3$ of order $n^{-1/2}$, up to logarithmic factors.

To state the main theorem of this section, we consider the simplified
setting in which $\lambda_s = \lambda_o = \lambda$. Remind that in
practice it is always recommended to normalize the columns of the
matrix $\bX$ so that their Euclidean norm is of the order $\sqrt{n}$.
The more precise version of the next result with better constants is
provided in the supplement (see \Cref{prop:augmented:analysis:v2}).
We recall that a matrix $\bfSigma$ is said to satisfy the restricted
eigenvalue condition $\RE(s,c_0)$ with some constant $\varkappa>0$,
if $\|\bfSigma^{1/2}\bv\|_2\ge \varkappa\|\bv_J\|_2$ for any
vector $\bv\in\mathbb R^p$ and any set $J\subset \{1,\ldots,p\}$ such
that  $\text{Card}(J)\le s$ and $\|\bv_{J^c}\|_1\le c_0\|\bv_J\|_1$.

\begin{theorem}\label{th:1}
Let $\bfSigma$ satisfy the $\RE(s,5)$ condition with constant $\varkappa>0$.
Let  $\sb_1$, $\sb_2$, $\sb_3$, $\sc_1$, $\sc_2$, $\sc_3$ be some positive
real numbers such that $\bfX$ satisfies the $\IP(0;\sb_2;\sb_3)$
and the $\ATP(\sc_1;\sc_2;\sc_3)$. Assume that for some $\delta\in(0,1)$,
the tuning parameter $\lambda$ satisfies
\begin{align}\label{lambda:1}
		\lambda\sqrt{n}\ge
		\sqrt{{8\log (n/\delta)}}\bigvee
		\big(\max_{j=1,\ldots,p}\|\bfX^{(n)}_{\bullet,j}\|_2\big)
		\sqrt{{8\log (p/\delta)}}.
\end{align}
If the sparsity $s$ and the number of outliers $o$ satisfy the condition
\begin{align}
		\frac{s}{\varkappa^2}+ o &\le
		\frac{\sc_1^2}{400\big({\sc_2}\vee {\sc_3}\vee 5\sb_2/\sc_1\big)^2},
		\label{cond1a}
\end{align}
then, with probability at least $1-2\delta$, we have
\begin{align}
		\big\|{\bfSigma}^{1/2}(\hat\bbeta-\bbeta^*)\big\|_2 \le
		\frac{24\lambda}{\sc_1^2}\Big(\frac{2\sc_2}{\sc_1}\bigvee
		\frac{\sb_3}{\sc_1^2}\Big)\Big(\frac{s}{\varkappa^2}+ 7o\Big)
		+ \frac{5\lambda\sqrt{s}}{6\sc_1^2\varkappa}.\label{riskbound:1}
\end{align}
\end{theorem}

\Cref{th:1} is somewhat hard to parse. At this stage, let us simply mention that in the case of a Gaussian
design considered in the next section, $\sc_1$ is of order $1$ while $\sb_2,\sb_3,\sc_2,\sc_3$ are of order
$n^{-1/2}$, up to a factor logarithmic in $p$, $n$ and $1/\delta$. Here $\delta$ is an upper bound on the
probability that the Gaussian matrix $\bfX$ does not satisfy either $\IP$ or $\ATP$. Since \Cref{th:1} allows
us to choose $\lambda$ of the order $\sqrt{\log\{(p+n)/\delta\}/n}$, we infer from \eqref{riskbound:1} that
the error of estimating $\bbeta^*$, measured in Euclidean norm, is of order $\frac{s}{n\varkappa^2} + \frac{o}{n}
+ (\frac{s}{n\varkappa^2})^{1/2} = O(\frac{o}{n} + (\frac{s}{n\varkappa^2})^{1/2})$, under the assumption that
$(\frac{s}{n\varkappa^2} + \frac{o}{n})\log(np/\delta)$ is smaller than a universal constant.

To complete this section, we present a sketch of the proof 
of \Cref{th:1}. In order to convey the main ideas without 
diving too much into technical details, we assume $\bfSigma 
=\bfI_p$. This means that the $\RE$ condition is satisfied 
with $\varkappa =1$ for any $s$ and $c_0$. From the fact 
that the $\ATP$ holds for $\bfX$, we infer that $[\bfX \,\sqrt{n}\,\bfI_n]$
satisfies the $\RE(s+o,5)$ condition with the constant $\sc_1/2$.  Using
the well-known risk bounds for the Lasso estimator \citep{BRT}, we get
\begin{align}
\|\hat\bbeta-\bbeta^*\|_2^2 + \|\hat\btheta-\btheta^*\|_2^2\le C \lambda^2(s+o)\quad\text{and}\quad
\|\hat\bbeta-\bbeta^*\|_1 + \|\hat\btheta-\btheta^*\|_1\le C \lambda(s+o).\label{slow:1}
\end{align}
Note that these are the risk bounds established in\footnote{the first two references deal with the small
dimensional case only, that is where $s=p\ll n$.} \citep{candes:randall2008,dalalyan:chen2012,nguyen:tran2013}.
These bounds are most likely unimprovable as long as the 
estimation of $\btheta^*$ is of interest. However, if we 
focus only on the estimation error of $\bbeta^*$, 
considering $\btheta^*$ as a nuisance parameter, the
following argument leads to a sharper risk bound. First, 
we note that
\begin{align}
    \hat\bbeta\in \text{arg}\min_{\bbeta\in\mathbb R^p} 
		\Big\{\frac1{2n}\|\bY-\bfX\bbeta-\sqrt{n}\,
    \hat\btheta\|_2^2+\lambda\|\bbeta\|_1\Big\}.
\end{align}
The KKT conditions of this convex optimization problem 
take the following form
\begin{align}
    \nicefrac1n\bfX^\top(\bY-\bfX\hat\bbeta-\sqrt{n}\, 
		\hat\btheta) \in \lambda\cdot \text{sgn}(\hat\bbeta),
\end{align}
where $\text{sgn}(\hat\bbeta)$ is the subset of $\mathbb R^p$ 
containing all the vectors $\bw$ such that $w_j\hat\beta_j 
= |\hat\beta_j|$ and $|w_j|\le 1$ for every $j\in\{1,\ldots,p\}$. 
Multiplying the last displayed equation from left by 
$\bbeta^*-\hat\bbeta$, we get
\begin{align}
    \nicefrac1n(\bbeta^*-\hat\bbeta)^\top\bfX^\top(\bY - 
		\bfX\hat\bbeta-\sqrt{n}\, \hat\btheta) \le
    \lambda\big(\|\bbeta^*\|_1-\|\hat\bbeta\|_1\big).
\end{align}
Recall now that $\bY = \bfX\bbeta^*+\sqrt{n}\,\btheta^*+\bxi$ 
and set $\bv = \bbeta^*-\hat\bbeta$ and
$\bu =\btheta^*-\hat\btheta$. We arrive at
\begin{align}
    \nicefrac1n\|\bfX\bv\|_2^2 = \nicefrac1n\bv^\top\bfX^\top\bfX\bv 
		\le -\bv^\top(\bfX^{(n)})^\top\bu - \nicefrac1n\bv^\top\bfX^\top\bxi
		+\lambda\big(\|\bbeta^*\|_1-\|\hat\bbeta\|_1\big).
\end{align}
On the one hand, the duality inequality and the lower bound on 
$\lambda$ imply that $|\bv^\top\bfX^\top\bxi|\le \|\bv\|_1
\|\bfX^\top\bxi\|_\infty \le n\lambda\|\bv\|_1/2$. On the other hand,
well-known arguments yield $\|\bbeta^*\|_1-\|\hat\bbeta\|_1\le 
2\|\bv_S\|_1-\|\bv\|_1$. Therefore, we have
\begin{align}\label{eq:c3}
    \nicefrac1n\|\bfX\bv\|_2^2 \le |\bv^\top(\bfX^{(n)})^\top\bu|
    +\nicefrac{\lambda}{2}\big(4\|\bv_S\|_1-\|\bv\|_1\big).
\end{align}
Since $\bX$ satisfies the $\ATPI(\sc_1,\sc_2,\sc_3)$ that implies 
the $\TPI(\sc_1,\sc_2)$, we get $\sc_1^2\|\bv\|_2^2\le \nicefrac2n
\|\bfX\bv\|_2^2 + 2\sc_2^2\|\bv\|_1^2$. Combining with \eqref{eq:c3},
this yields
\begin{align}
    \sc_1^2\|\bv\|_2^2
        &\stackrel{\hphantom{\IPI(0,\sb_2,\sb_3)}}{\le} 2|\bv^\top(\bfX^{(n)})^\top\bu| +
        {\lambda}\big(4\|\bv_S\|_1-\|\bv\|_1\big)+2\sc_2^2\|\bv\|_1^2\\
        &\stackrel{\IPI(0,\sb_2,\sb_3)}{\le}
        2\sb_3\|\bv\|_2\|\bu\|_1 + 2\sb_2\|\bv\|_1\|\bu\|_2 + {\lambda}\big(4\|\bv_S\|_1-\|\bv\|_1\big)+2\sc_2^2\|\bv\|_1^2
        \label{eqq}\\
        &\stackrel{\hphantom{\IPI(0,\sb_2,\sb_3)}}{\le}
        \frac{\sc_1^2}{2}\|\bv\|_2^2+ \frac{2\sb_3^2}{\sc_1^2}\|\bu\|_1^2 + \|\bv\|_1(2\sb_2\|\bu\|_2-\lambda) + 4{\lambda}\|\bv_S\|_1
        +2\sc_2^2\|\bv\|_1^2.
\end{align}
Using the first inequality in \eqref{slow:1} and condition \eqref{cond1a}, we upper bound
$(2\sb_2\|\bu\|_2-\lambda)$ by 0. To upper bound the second last term, we use the Cauchy-Schwarz
inequality: $4{\lambda}\|\bv_S\|_1 \le 4\lambda\sqrt{s}\,\|\bv\|_2\le (4/\sc_1)^2 \lambda^2s +
(\sc_1/2)^2\|\bv\|_2^2$. Combining all these bounds and rearranging the terms, we arrive at
\begin{align}
    (\sc_1^2/4)\|\bv\|_2^2
        &\le
        2 \{(\sb_3/\sc_1)\vee \sc_2\}^2 (\|\bu\|_1 +\|\bv\|_1)^2 + (4/\sc_1)^2 \lambda^2 s.
\end{align}
Taking the square root of both sides and using the second inequality
in \eqref{slow:1}, we obtain an inequality of the same type as
\eqref{riskbound:1} but with slightly larger constants. As a
concluding remark for this sketch of proof, let us note that
if instead of using the last arguments, we replace all the
error terms appearing in \eqref{eqq} by their upper bounds
provided by \eqref{slow:1}, we do not get the optimal rate.

\section{The case of Gaussian design}
\label{sec:3}

Our main result, \Cref{th:1}, shows that if the design
matrix satisfies the transfer principle and the incoherence
property with suitable constants, then the $\ell_1$-penalized
Huber's $M$-estimator achieves the optimal rate under
adversarial contamination. As a concrete example of a design
matrix for which the aforementioned conditions are satisfied,
we consider the case of correlated Gaussian design. As opposed
to most of prior work on robust estimation for linear regression
with Gaussian design, we allow the covariance matrix to have a
non degenerate null space. We will simply assume that the $n$
rows of the matrix $\bfX$ are independently drawn from the
Gaussian distribution $\mathcal N_p(\mathbf 0,\bfSigma)$ with
a covariance matrix $\bfSigma$ satisfying the $\RE(s,5)$
condition. We will also assume in this section that all the
diagonal entries of $\bfSigma$ are equal to 1: $\bfSigma_{jj}=1$.
The more formal statements of the results, provided in the
supplementary material, do not require this condition.

\begin{theorem}\label{th:2}
Let $\delta\in(0,1/7)$ be a tolerance level and $n\ge 100$.
For every positive semi-definite matrix $\bfSigma$ with all
the diagonal entries bounded by one, with probability at least
$1-2\delta$, the matrix $\bfX$ satisfies the $\TP(\sa_1,\sa_2)$,
the $\IP(\sb_1,\sb_2,\sb_3)$ and the $\ATP(\sc_1,\sc_2,\sc_3)$
with constants
\begin{align}
\sa_1 &= 1 - \frac{4.3 + \sqrt{2\log(9/\delta)}}{\sqrt{n}} ,\qquad
\sa_2 = \sb_2 = 1.2\sqrt{\frac{2\log p}{n}}\\
\sb_1 &= \frac{4.8\sqrt{2} + \sqrt{2\log(81/\delta)}}{\sqrt{n}} ,\qquad
\sb_3 = 1.2\sqrt{\frac{2\log n}{n}},\\
\sc_1 &= \frac34 - \frac{17.5 + 9.6\sqrt{2\log(2/\delta)}}{\sqrt{n}} ,\qquad
\sc_2 = 3.6\sqrt{\frac{2\log p}{n}},\qquad
\sc_3 = 2.4\sqrt{\frac{2\log n}{n}}.
\end{align}
\end{theorem}

The proof of this result is provided in the supplementary
material. It relies on by now standard tools such as Gordon's
comparison inequality, Gaussian concentration inequality and
the peeling argument. Note that the $\TP$ and related results
have been obtained in \cite{raskutti:wainwright:yu2010,oliveira2016,rudelson:zhou2013}.
The $\IP$ is basically a combination of a high probability
version of Chevet's inequality \citep[Exercises 8.7.3-4]{Vershynin2018}
and the peeling argument. A property similar to the $\ATP$
for Gaussian matrices with non degenerate covariance was
established in \citep[Lemma 1]{nguyen:tran2013} under further
restrictions on $n,p,s,o$.

\begin{theorem}\label{th:3}
There exist universal positive constants $\sd_1$, $\sd_2$, $\sd_3$ 
such that if
$$
\frac{s\log p}{\varkappa^2} +  o\log n \le \sd_1 n\qquad \text{and}\qquad
1/7\ge \delta\ge 2e^{-\sd_2 n}
$$
then, with probability at least $1-4\delta$, $\ell_1$-penalized Huber's
$M$-estimator with $\lambda_s^2 n = 9\sigma^2\log(p/\delta)$ and 
$\lambda_o^2n = 8\sigma^2\log(n/\delta)$ satisfies
\begin{align}
	\big\|{\bfSigma}^{1/2}(\hat\bbeta-\bbeta^*)\big\|_2 \le\sd_3
	\sigma
	\bigg\{\Big(\frac{s\log (p/\delta)}{n\varkappa^2}\Big)^{1/2} +
    \frac{o\log(n/\delta)}{n}\bigg\}.\label{riskbound:2}
\end{align}
\end{theorem}

Even though the constants appearing in \Cref{th:2} are reasonably
small and smaller than in the analogous results in prior work, the
constants $\sd_1$, $\sd_2$ and $\sd_3$ are large, too large for
being of any practical relevance. Finally, let us note that if $s$
and $o$ are known, it is very likely that following the techniques
developed in \citep[Theorem 4.2]{bellec2018}, one can replace the
terms $\log(p/\delta)$ and $\log(n/\delta)$ in \eqref{riskbound:2}
by $\log(p/s\delta)$ and $\log(n/o\delta)$, respectively.

Comparing \Cref{th:3} with \citep[Theorem 1]{nguyen:tran2013}, we see
that our rate improvement is not only in terms of its dependence on
the proportion of outliers, $o/n$, but also in terms of the condition
number $\varkappa$, which is now completely decoupled from $o$ in the
risk bound.

While our main focus is on the high dimensional situation
in which $p$ can be larger than $n$, it also applies to the case
of small dimensional dense vectors, \textit{i.e.}, when $s=p$
is significantly smaller than $n$. One of the applications of
such a setting is the problem of stylized communication considered,
for instance, in \citep{candes:randall2008}. The problem is to
transmit a signal $\bbeta^*\in\mathbb R^p$ to a remote receiver.
What the receiver gets is a linearly transformed codeword $\bfX\bbeta^*$
corrupted by small noise and malicious errors. While all the entries
of the received codeword are affected by noise, only a fraction of them
is corrupted by malicious errors, corresponding to outliers. The receiver
has access to the corrupted version of $\bfX\bbeta^*$ as well as
to the encoding matrix $\bfX$. Theorem 3.1 from \citep{candes:randall2008}
establishes that the Dantzig selector \citep{candes:tao2007}, for a
properly chosen tuning parameter proportional to the noise level, achieves the
(sub-optimal) rate $\sigma^2(s+o)/n$, up to a logarithmic factor.
A similar result, with a noise-level-free version of the Dantzig
selector, was proved in \citep{dalalyan:chen2012}. Our \Cref{th:3}
implies that the error of the $\ell_1$-penalized Huber's estimator
goes to zero at the faster rate $\sigma^2\{(s/n)+(o/n)^2\}$. 

Finally, one can deduce from \Cref{th:3} that as soon as the number
of outliers satisfies $o=o(\sqrt{sn/\varkappa^2})$, the rate of convergence
remains the same as in the outlier-free setting.

\section{Prior work}
\label{sec:3b}

As attested by early references such as \citep{Tukey}, robust estimation
has a long history. A remarkable---by now classic---result by
\cite{huber1964} shows that among all the shift invariant $M$-estimators
of a location parameter, the one that minimizes the asymptotic variance
corresponds to the loss function $\phi(x) = 1/2 \{x^2\wedge (2x-1)\}$.
This result was proved in the case when the reference distribution is
univariate Gaussian. Apart from some exceptions, such as \citep{yatracos1985},
during several decades the literature on robust estimation was mainly
exploring the notions of breakdown point, influence function, asymptotic
efficiency, etc., see for instance \citep{donoho1992,hampel2005robust,
huber:ronchetti2009} and the recent survey \citep{Yu2017}. A more recent 
trend in statistics is to focus on finite sample risk bounds that are 
minimax-rate-optimal when the sample size $n$, the dimension $p$ of 
the unknown parameter and  the number $o$ of outliers tend jointly 
to infinity \citep{chen:gao:ren2018, chen2016,Gao2017}.

In the problem of estimating the mean of a multivariate
Gaussian distribution, it was shown that the optimal rate of the
estimation error measured in Euclidean norm scales as $(p/n)^{1/2}
+ (o/n)$. Similar results were established for the problem of robust
linear regression as well. However, the estimator that was shown to
achieve this rate under fairly general conditions on the design is
based on minimizing regression depths, which is a hard computational
problem. Several alternative robust estimators with polynomial 
complexity were proposed 
\citep{diakonikolasKKLMS2016,lai:rao:vempala2016,Cheng2019,Collier2017,DiakonikolasKK018}.

Many recent papers studied robust linear regression. 
\citep{karmalkar:price2018} considered $\ell_1$-constrained minimization
of the $\ell_1$-norm of residuals and found a sharp threshold on the
proportion of outliers determining whether the error of estimation tends
to zero or not, when the noise level goes to zero. From a methodological 
point of view, $\ell_1$-penalized Huber's estimator has been considered 
in \citep{She2011,lee2012}. These papers contain also comprehensive 
empirical evaluation and proposals for data-driven choice of tuning 
parameters. Robust sparse regression with an emphasis on contaminated
design was investigated in \citep{pmlr-v28-chen13h,
balakrishnan:du:li:singh2017,DiakonikolasKS19,Liu18,Liu19}.  
Iterative and adaptive hard thresholding approaches were considered
in \citep{Bhatia2015,Bhatia2017,Bhatia2019}. Methods based on penalizing the 
vector of outliers were studied by \cite{Li2013,Foygel14,Adcock}, who 
adopted a more signal-processing point of view in which the noise vector 
is known to have a small $\ell_2$ norm and nothing else is known 
about it. We should stress that our proof techniques share many 
common features with those in \citep{Foygel14}. See also \cite{sardy:tseng:bruce2001}.

The problem of robust estimation of graphical
models, closely related to the present work, was addressed 
in \citep{balmand:dalalyan2015,Katiyar,Liu19}. Quite surprisingly, 
at least to us, the minimax rate of robust estimation
of the precision matrix in Frobenius norm is not known yet.  

\section{Extensions}
\label{sec:4}

The results presented in previous sections pave the
way for some future investigations, that are discussed
below. None of these extensions is carried out in this
work, they are listed here as possible avenues for
future research.

\paragraph{Contaminated design}

In addition to labels, the features also might be
corrupted by outliers. This is the case, for instance,
in Gaussian graphical models. Formally, this means that
instead of observing the clean data $\{(\bX_i^\circ,
y_i^\circ);i=1,\ldots,n\}$ satisfying $y_i^\circ =
(\bX_i^\circ)^\top \beta^* + \xi_i$, we observe
$\{(\bX_i,y_i);i=1,\ldots,n\}$ such that $(\bX_i,y_i)
= (\bX_i^\circ,y_i^\circ)$ for all $i$ except for
a fraction of outliers $i\in O$. In such a setting,
we can set $\theta_i^* = (y_i-\bX_i^\top \bbeta^*
-\xi_i)/\sqrt{n}$ and recover exactly the same
model as in \eqref{lm:2}.

The important difference as compared to the setting
investigated in previous section is that it is not
reasonable anymore to assume that the feature vectors
$\{\bX_i: i\in O\}$ are iid Gaussian. In the adversarial
setting, they may even be correlated with the noise
vector $\bxi$. It is then natural to remove all the
observations for which $\max_{j} |\bX_{ij}|>
\sqrt{2\log np/\delta}$ and to assume, that the
$\ell_1$-penalized Huber estimator is applied to data
for which $\max_{ij} |\bX_{ij}|\le
\sqrt{2\log np/\delta}$. This implies that $\lambda$ can
be chosen of the order of\footnote{We use notation
$a_n =\tilde O(b_n)$ as a shorthand for $a_n\le C b_n
\log^c n$ for some  $C,c>0$ and for every $n$.}
$\sigma\tilde O(n^{-1/2} + (o/n))$, which is an
upper bound on $\|\bX^\top\bxi\|_\infty/n$.

In addition, $\TP$ is clearly satisfied since it is
satisfied for the submatrix  $\bfX_{O^c}$ and $\|\bfX\bv\|_2
\ge \|\bfX_{O^c}\bv\|_2$. As for the $\IP$, we know from
\Cref{th:2} that $\bfX_{O^c}$ satisfies $\IP$ with constants
$\sb_1$, $\sb_2$, $\sb_3$ of order $\tilde O(n^{-1/2})$.
On the other hand,
\begin{align}
|\bu_O^\top\bfX_{O}\bv| \le \|\bfX\|_\infty\|\bu_O\|_1\|\bv\|_1
\le \sqrt{2o\log(np/\delta)} \|\bu_O\|_2\|\bv\|_1.
\end{align}
This implies that $\bfX$ satisfies $\IP$ with
$\sb_1=\tilde O(n^{-1/2})$, $\sb_2 = \tilde O((o/n)^{1/2})$
and $\sb_3=\tilde O(n^{-1/2})$. Applying \Cref{th:1}, we
obtain that if $(so+o^2)\log(np) \le cn$ for a sufficiently
small constant $c>0$, then with high probability
\begin{align}
	\|\bfSigma^{1/2}(\hat\bbeta-\bbeta^*)\|_2 = \sigma\tilde
	O\bigg\{\sqrt{\frac{s}{n}} + \frac{o\sqrt{s}}{n} +
	\sqrt{\frac{o}{n}}\Big(\frac{1}{\sqrt{n}}+\frac{o}{n}
	\Big)(s+o)\bigg\} = \sigma O\bigg\{\sqrt{\frac{s}{n}} +
	\frac{\sqrt{o^3}}{n}\bigg\}.
\end{align}
This rate of convergence appear to be slower than those
obtained by methods tailored to deal with corruption in
design, see  \Citep{Liu18,Liu19} and the references therein.
Using more careful analysis, this rate might be improvable.
On the positive side, unlike many of its competitors, the
estimator $\hat\bbeta$ has the advantage of being independent
of the covariance matrix $\bfSigma$ and on the sparsity $s$.
Furthermore, the upper bound does not depend, even
logarithmically, on $\|\bbeta^*\|_2$. Finally, if $o^3\le sn$,
our bound yields the minimax-optimal rate. To the best of our
knowledge, none of the previously studied robust estimators
has such a property.

\paragraph{Sub-Gaussian design}
The proof of \Cref{th:2} makes use of some results,
such as Gordon-Sudakov-Fernique or Gaussian concentration
inequality, which are specific to the Gaussian
distribution. A natural question is whether the rate
$\sigma\{(\frac{s\log(p/s)}{n})^{1/2} + \frac{o}{n}\}$
can be obtained for more general design distributions.
In the case of a sub-Gaussian design with the scale-
parameter $1$, it should be possible to adapt the
methodology developed in this work to show that the
$\TP$ and the $\IP$ are satisfied with high-probability.
Indeed, for proving the $\IP$, it is possible to replace
Gordon's comparison inequality by Talagrand's sub-Gaussian
comparison inequality \citep[Cor.~8.6.2]{Vershynin2018}.
The Gaussian concentration inequality can be replaced by
generic chaining.

\paragraph{Heavier tailed noise distributions}

For simplicity, we assumed in the paper that the random variables
$\xi_i$ are drawn from a Gaussian distribution. As usual for the
Lasso analysis, all the results extend to the case of sub-Gaussian
noise, see \citep{koltchinskii2011oracle}. Indeed, we only
need to control tail probabilities of the random variable
$\|\bfX^\top\bxi\|_\infty$ and $\|\bxi\|_\infty$, which can be
done using standard tools. We believe that it is possible to
extend our results beyond sub-Gaussian noise, by assuming some
type of heavy-tailed distributions. The rationale behind this is
that any random variable $\xi$ can be written (in many different
ways) as a sum of a sub-Gaussian variable $\xi^{\text{noise}}$ and
a ``sparse'' variable $\xi^{\text{out}}$. By ``sparse'' we mean
that $\xi^{\text{out}}$ takes the value 0 with high probability.
The most naive way for getting such a decomposition is to set
$\xi^{\text{noise}} = \xi\mathds 1(|\xi|<\tau)$ and
$\xi^{\text{out}} = \xi\mathds 1(|\xi|\ge \tau)$. The random
noise terms $\xi_i^{\text{out}}$ can be merged with $\theta_i$ and
considered as outliers. We hope that this approach can establish
a connection between two types of robustness: robustness to
outliers considered in this work and robustness to heavy tails
considered in many recent papers
\citep{Devroye,Catoni,Minsker18,Lugosi,2017arXiv171110306L}.



\section{Numerical illustration}

We performed a synthetic experiment to illustrate the 
obtained theoretical result and to check that it is in line 
with numerical results. We chose $n=1000$ and $p=100$ for 
3 different levels of sparsity $s=5,15,25$. The noise 
variance was set to $1$ and $\bbeta^*$ was set to have 
its first $s$ non-zero coordinates equal to 10. Each corrupted 
response coordinate was $\theta^*_j=10$. The fraction 
$\epsilon=o/n$ of outliers was ranging between 0 and 0.25 with 
a step-size of 5 for the number of outliers $o$ is used.  The
MSE was computed using 200 independent repetitions. The 
optimisation problem in \eqref{Lasso} was solved using the 
\texttt{glmnet} package with the tuning parameters 
$\lambda_s=\lambda_o = \sqrt{(8/n)(\log(p/s)+\log(n/o))}$. 
\centerline{\includegraphics[width=0.65\textwidth]{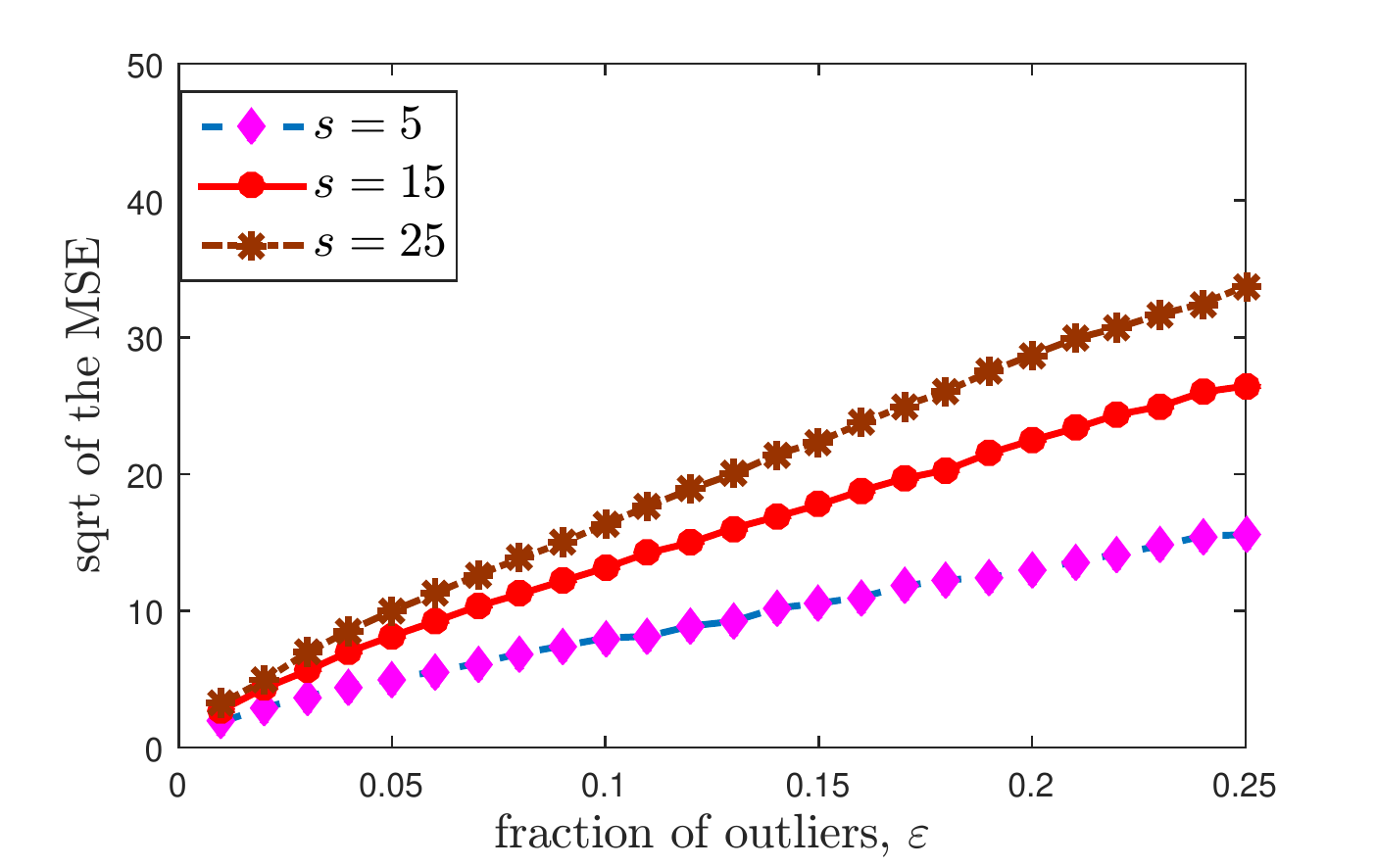}}
The obtained plots clearly demonstrate that there is a linear
dependence on $\varepsilon$ of the square-root of the mean
squared error. 

\section{Conclusion}
\label{sec:5}

We provided the first proof of the rate-optimality---up to
logarithmic terms that can be avoided---of $\ell_1$-penalized
Huber's $M$-estimator in the setting of robust linear regression
with adversarial contamination. We established this result under
the assumption that the design is Gaussian with a covariance matrix
$\bfSigma$ that need not be invertible. The condition number
governing the risk bound is the ratio of the largest diagonal
entry of $\Sigma$ and its restricted eigenvalue. Thus, in addition
to improving the rate of convergence, we also relaxed the
assumptions on the design. Furthermore, we outlined some
possible extensions, namely to corrupted design and/or
sub-Gaussian design, which seem to be fairly easy to carry
out building on the current work.

Next on our agenda is the more thorough analysis of the robust
estimation by $\ell_1$-penalization in the case of contaminated
design. A possible approach, complementary to the one described
in \Cref{sec:4} above, is to adopt an errors-in-variables point
of view similar to that developed in \citep{belloni2016}.
Another interesting avenue for future research is the development
of scale-invariant robust estimators and their adaptation to the
Gaussian graphical models. This can be done using methodology
brought forward in \citep{JMLR:v14:sun13a,balmand:dalalyan2015}.
Finally, we would like to better understand what is the largest
fraction of outliers for which the $\ell_1$-penalized Huber's
$M$-estimator has a risk---measured in Euclidean norm---upper
bounded by $\sigma o/n$. Answering this question even under
stringent assumptions of independent standard Gaussian design
$\bX_{ij}$ with $(s\log p)/n$ going to zero as $n$ tends to
infinity would be of interest.

\newpage

\section{Acknowledgements}

We would like to thank the reviewers for the careful reading of the paper and for helpful and thoughtful remarks. This work was supported by LabEx ECODEX and a public grant as part of the Investissement d'avenir project, reference ANR-11-LABX-0056-LMH, LabEx LMH.

\bibliographystyle{apalike}
\bibliography{Robust_SLM}

\newpage
\part*{Supplementary material}

The theorems stated in the paper are consequences of
\Cref{prop:improved:rate}, \Cref{pr:gen:incoherence}
and \Cref{pr:transfer:principle:1dim}. These results are proved
in subsequent sections, which are organized as follows.
\Cref{sec:pr1} contains tight risk bounds for general matrices
satisfying the transfer principle and the incoherence property.
We then show in \Cref{sec:pr2} that the Gaussian design satisfies,
with high probability, both the transfer principle and the
incoherence property. We complete the paper by showing how
\Cref{th:1},  \Cref{th:2} and \Cref{th:3} can be deduced from
\Cref{prop:improved:rate}, \Cref{pr:gen:incoherence}
and \Cref{pr:transfer:principle:1dim}.

To help the reader to navigate through the proof without losing
the thread, the diagram below outlines the relations between
different auxiliary results.

\smallskip
\def\svgwidth{\textwidth}
\centerline{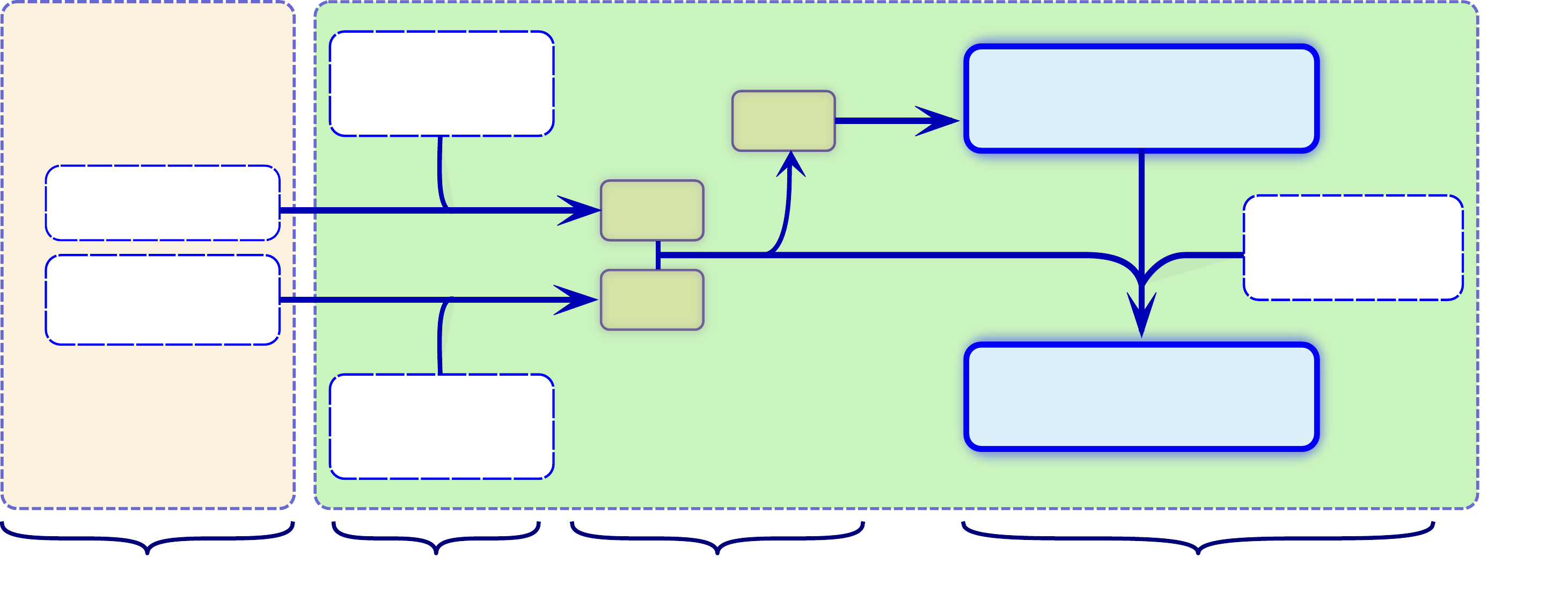}
\smallskip

Thus, \Cref{prop:augmented:analysis:v2} establishes a risk bound valid under
$\ATP$. This risk bound is sub-optimal for Gaussian designs, but it is an
intermediate step for getting the final risk bound, established in
\Cref{prop:improved:rate}. The latter follows from the $\TP$, $\IP$
and an auxiliary result proved in \Cref{lemma:aux1}. The fact that the $\TP$
holds true for Gaussian matrices is proved in \Cref{pr:transfer:principle:1dim}
as a consequence of \Cref{lemma:aux1} and one-parameter peeling
(\Cref{lemma:peeling:1dim}). Similarly, the fact that the $\IP$
holds true for Gaussian matrices is proved in \Cref{pr:gen:incoherence}
as a consequence of \Cref{lem:GW1} and two-parameter peeling
(\Cref{lemma:peeling:2dim}).

\section{Main technical results for general design matrices}\label{sec:pr1}

In the sequel, we denote by $\mathbb{S}^{k-1}$ the unit sphere
in $\re^k$ with respect to the Euclidean  norm  centered at
the origin. With a slight abuse of notation, $\re^k$ will be
identified with $\re^{k\times1}$. The unit ball with respect
to the $\ell_p$-norm centered at the origin will be denoted by
$\mathbb{B}_p^k$. Given a matrix $\bfSigma\in\re^{p\times p}$,
we will use the definition
$
\rho(\bfSigma):=\max_{j\in[p]}\sqrt{\bfSigma{}_{jj}}
$
without further notice. We will use notation $\bfDelta^{\bbeta} =
\hat\bbeta-\bbeta^*$, $\bfDelta^{\btheta} = \hat\btheta-\btheta^*$ and
$\bfDelta = [\bfDelta^{\bbeta};\bfDelta^{\btheta}]\in \mathbb{R}^{p+n}$.
We denote by $S$ the support of $\bbeta^*$ and by $O$ that of $\btheta^*$.
We know that $\text{Card}(S)\le s$ and $\text{Card}(O)\le o$.
Throughout, we set $\gamma =\lambda_s/\lambda_o$ and define the
dimension reduction cone $\mathcal C_{S,O}(c_0,\gamma)
=\{(\bu,\bv)\in\mathbb R^n\times\mathbb R^p: \|\bu_{O^c}\|_1 + \gamma
\|\bv_{S^c}\|_1\le c_0(\|\bu_{O}\|_1 + \gamma\|\bv_{S}\|_1)\}$, where
$c_0\ge 1$ is a constant.

\subsection{Augmented transfer principle implies the sub-optimal rate}

This section is devoted to the proof of the fact that the estimators
$\hat\bbeta$ and $\hat\btheta$ achieve, up to logarithmic factors, the
rates
$$
\frac{s}{n\varkappa^2} + \frac{o}{n} \qquad\text{and}\qquad
\frac{s}{\sqrt{n}\varkappa^2} + \frac{o}{\sqrt{n}}
$$
for squared $\ell_2$ error and $\ell_1$ errors, respectively. This is true
under suitable conditions on the design matrix $\bfX$. These rates are
not optimal, but they will help us to obtain the optimal rates.

\begin{proposition}\label{prop:augmented:analysis:v2}
Let $\bfSigma$ satisfy the $\RE(s,5)$ with constant $\varkappa>0$.
Let $\sc_1,\sc_2,\sc_3$ and $\gamma$ be some positive real
numbers satisfying
\begin{align}
		8\big({\sc_2}\vee \gamma\sc_3\big)
		\bigg(\frac{s}{\varkappa^2}+ \frac{6.25o}{\gamma^2}\bigg)^{1/2} \le \sc_1.
\end{align}
Assume that on some event $\Omega$, the following conditions are met:
\begin{itemize}\itemsep=0pt
\item[\rm(i)] $\bfX$ satisfies the $\ATP\left(\sc_1;\sc_2;\sc_3\right)$ .
\item[\rm(ii)]
$
\lambda_s = \gamma\lambda_o
\ge (\nicefrac{2}{n})\Vert\bfX^\top\bxi\Vert_\infty,\quad\text{and}\quad
\lambda_o \ge (\nicefrac{2}{\sqrt{n}})\Vert\bxi\Vert_\infty.
$
\end{itemize}
Then, on the same event $\Omega$, we have $\bfDelta\in\mathcal{C}_{S,O}
(3,\lambda_s/\lambda_o)$ and
\begin{align}
	\big\|\bfSigma^{1/2}\bfDelta^{\bbeta}\|_2^2 + \|\bfDelta^{\btheta}\big\|_2^2
	&\le \frac{36}{\sc_1^4}
	\bigg(\frac{\lambda_s^2 s}{\varkappa^2}+ 6.25\lambda_o^2o\bigg),
\label{equation:aux:rate:l2:v2}\\
\lambda_s\big\|\bfDelta^{\bbeta}\|_1 + \lambda_o \|\bfDelta^{\btheta}\big\|_1	
	& \le \frac{24}{\sc_1^2}\bigg(\frac{\lambda_s^2 s}{\varkappa^2}+ 6.25\lambda_o^2o\bigg)
	.\label{equation:aux:rate:l1:v2}
\end{align}
\end{proposition}
\begin{proof}
First, we use the
KKT conditions to infer that for some vectors $\bu\in\mathbb{B}^n_{\infty}$ and $\bv\in
\mathbb{B}^p_{\infty}$ such that $\bu^\top\hat\btheta = \|\hat\btheta\|_1$ and
$\bv^\top\hat\bbeta = \|\hat\bbeta\|_1$, we have
\begin{align}
[\bfX^{(n)} \, \bfI_n]^\top\big( \by^{(n)} - \bfX^{(n)}\hat\bbeta - \hat\btheta  \big)
 = [\lambda_s \bv ; \lambda_o \bu].
\end{align}
Using the facts that $\by^{(n)} = \bfX^{(n)}\bbeta^*+\btheta^* + \bxi^{(n)}$ and
rearranging the terms,  the last display takes the form
\begin{align}
[\bfX^{(n)} \, \bfI_n]^\top[\bfX^{(n)} \, \bfI_n]\bfDelta
 = [(\bfX^{(n)})^\top\bxi^{(n)}\,;\,\bxi^{(n)}]+ [\lambda_s \bv; \lambda_o \bu].
\end{align}
Multiplying the last display from the left by $\bfDelta^\top$, we arrive
at
\begin{align}
\|[\bfX^{(n)} \, \bfI_n]\bfDelta\|_2^2
 = (\bfDelta^{\bbeta})^\top(\bfX^{(n)})^\top\bxi^{(n)} +
(\bfDelta^{\btheta})^\top\bxi^{(n)}+ \lambda_s (\bfDelta^{\bbeta})^\top\bv +
\lambda_o (\bfDelta^{\theta})^\top\bu.
\end{align}
The relations $\|\bv\|_\infty\le 1$ and $\bv^\top\hat\bbeta = \|\hat\bbeta\|_1$
imply that $(\bfDelta^{\bbeta})^\top\bv = (\bbeta^*-\hat\bbeta)^\top\bv =
(\bbeta^*)^\top\bv -\|\hat\bbeta\|_1\le \|\bbeta^*\|_1-\|\hat\bbeta\|_1$.
Similarly, $(\bfDelta^{\btheta})^\top\bu \le \|\btheta^*\|_1-\|\hat\btheta\|_1$.
Combining these bounds with the duality inequality and the last display, we infer
that
\begin{align}
\|[\bfX^{(n)} \, \bfI_n]\bfDelta\|_2^2
 &\le  \|\bfDelta^{\bbeta}\|_1\|(\bfX^{(n)})^\top\bxi^{(n)}\|_\infty +
  \|\bfDelta^{\btheta}\|_1\|\bxi^{(n)}\|_\infty\\
 &\qquad + \lambda_s \big(\|\bbeta^*\|_1 -\|\hat\bbeta\|_1\big) +
  \lambda_o \big(\|\btheta^*\|_1 -\|\hat\btheta\|_1\big)\\
 &\stackrel{\rm (ii)}{\le}	
		(\nicefrac{\lambda_s}{2})\|\bfDelta^{\bbeta}\|_1 +
    (\nicefrac{\lambda_o}{2})\|\bfDelta^{\btheta}\|_1
		+ \lambda_s \big(\|\bbeta^*\|_1 -\|\hat\bbeta\|_1\big) +
  \lambda_o \big(\|\btheta^*\|_1 -\|\hat\btheta\|_1\big).
	\label{eq:25}
\end{align}
Recall that $J=\{j:\bbeta_j\neq 0\}$ and $O = \{i:\btheta^*_i\neq 0\}$.
We have
\begin{align}
	\|\bfDelta^{\bbeta}\|_1+2\|\bbeta^*\|_1-2\|\hat\bbeta\|_1
	&= \|\bfDelta^{\bbeta}\|_1+2\|\bbeta^*_S\|_1-2\|\hat\bbeta_S\|_1
			-2\|\bfDelta^{\bbeta}_{S^c}\|_1\\
	&\le \|\bfDelta^{\bbeta}\|_1+2\|\bfDelta^{\bbeta}_S\|_1
			-2\|\bfDelta^{\bbeta}_{S^c}\|_1\\
	& = 3\|\bfDelta^{\bbeta}_S\|_1 -\|\bfDelta^{\bbeta}_{S^c}\|_1.
\end{align}
The same type of reasoning leads to $\|\bfDelta^{\btheta}\|_1+
2\|\btheta^*\|_1-2\|\hat\btheta\|_1 \le 3\|\bfDelta^{\btheta}_O\|_1
-\|\bfDelta^{\btheta}_{O^c}\|_1$. Combining these inequalities with
\eqref{eq:25}, we get
\begin{align}
\|[\bfX^{(n)} \, \bfI_n]\bfDelta\|_2^2
 &\le	
		(\nicefrac{\lambda_s}{2})\big(3\|\bfDelta^{\bbeta}_S\|_1 -
		\|\bfDelta^{\bbeta}_{S^c}\|_1\big) +
    (\nicefrac{\lambda_o}{2})\big(3\|\bfDelta^{\btheta}_O\|_1 -
		\|\bfDelta^{\btheta}_{O^c}\|_1\big).
	\label{eq:26}
\end{align}
On the one hand, since the left hand side is non negative, this obviously implies that
the vector $\bfDelta$ belongs to the dimension reduction cone
$\mathcal C_{S,O}(3,\gamma)$. On the other hand, using the $\ATP$,
\begin{align}
\sc_1 \big\|[\bfSigma^{1/2}\bfDelta^{\bbeta} \,;\, \bfDelta^{\btheta}]\big\|_2&-
		\sc_2\|\bfDelta^{\bbeta}\|_1 - \sc_3\|\bfDelta^{\btheta}\|_1\\
 &\le	
		\sqrt{(\nicefrac{\lambda_s}{2})\big(3\|\bfDelta^{\bbeta}_S\|_1 -
		\|\bfDelta^{\bbeta}_{S^c}\|_1\big) +
    (\nicefrac{\lambda_o}{2})\big(3\|\bfDelta^{\btheta}_O\|_1 -
		\|\bfDelta^{\btheta}_{O^c}\|_1\big)}.
	\label{eq:27}
\end{align}
We split the rest of the proof into two parts:  the first corresponds to
the case $5\|\bfDelta^{\bbeta}_S\|_1 \ge \|\bfDelta^{\bbeta}_{S^c}\|_1$
while the second treats the case $5\|\bfDelta^{\bbeta}_S\|_1 \le
\|\bfDelta^{\bbeta}_{S^c}\|_1$. The main goal of this splitting is to avoid
imposing strong assumption on $\bfSigma$ such as $\sigma_{\min}(\bfSigma)>0$
and to use the RE condition only.
\begin{description}
\item[Case 1:]$5\|\bfDelta^{\bbeta}_S\|_1 \ge \|\bfDelta^{\bbeta}_{S^c}\|_1$.
This is the simple case, since we know that $\bfDelta^{\bbeta}$ lies in the
suitable dimension reduction cone for which we can use the RE condition.
We first use the already proved fact $\bfDelta\in\mathcal C_{S,O}(3,
\gamma)$ to infer that
\begin{align}
\sc_2\|\bfDelta^{\bbeta}\|_1 + \sc_3\|\bfDelta^{\btheta}\|_1 &\le
\bigg(\frac{\sc_2}{\lambda_s}\bigvee \frac{\sc_3}{\lambda_o}\bigg)\big(
\lambda_s\|\bfDelta^{\bbeta}\|_1 + \lambda_o\|\bfDelta^{\btheta}\|_1)\\
&\le 4
\bigg(\frac{\sc_2}{\lambda_s}\bigvee \frac{\sc_3}{\lambda_o}\bigg)\big(
\lambda_s\|\bfDelta^{\bbeta}_S\|_1 + \lambda_o\|\bfDelta^{\btheta}_O\|_1)\\
&\le 4
\bigg(\frac{\sc_2}{\lambda_s}\bigvee \frac{\sc_3}{\lambda_o}\bigg)\bigg(
\frac{\lambda_s^2 s}{\varkappa^2}+ \lambda_o^2o\bigg)^{1/2}
(\varkappa^2\|\bfDelta^{\bbeta}_S\|_2^2 + \|\bfDelta^{\btheta}_O\|_2^2)^{1/2}\\
&\le 4
\bigg(\frac{\sc_2}{\lambda_s}\bigvee \frac{\sc_3}{\lambda_o}\bigg)
\bigg(\frac{\lambda_s^2 s}{\varkappa^2}+ \lambda_o^2o\bigg)^{1/2}
\big\|[\bfSigma^{1/2}\bfDelta^{\bbeta} \,;\, \bfDelta^{\btheta}]\big\|_2.
\label{eq:28}
\end{align}
Similarly, the right hand side of \eqref{eq:27} can be bounded by the square-root
of the expression
\begin{align}
	3(\nicefrac{\lambda_s}{2})\|\bfDelta^{\bbeta}_S\|_1 +
    3(\nicefrac{\lambda_o}{2})\|\bfDelta^{\btheta}_O\|_1
	&\le 1.5\bigg(\frac{\lambda_s^2 s}{\varkappa^2}+ \lambda_o^2o\bigg)^{1/2}
		(\varkappa^2\|\bfDelta^{\bbeta}_S\|_2^2 + \|\bfDelta^{\btheta}_O\|_2^2)^{1/2}\\
	&\le 1.5\bigg(\frac{\lambda_s^2 s}{\varkappa^2}+ \lambda_o^2o\bigg)^{1/2}
	\big\|[\bfSigma^{1/2}\bfDelta^{\bbeta} \,;\, \bfDelta^{\btheta}]\big\|_2.\label{eq:29}
\end{align}
To ease notation, we define $A =4
\big(\frac{\sc_2}{\lambda_s}\bigvee \frac{\sc_3}{\lambda_o}\big)
\big(\frac{\lambda_s^2 s}{\varkappa^2}+ \lambda_o^2o\big)^{1/2}$, $B = 1.5
\big(\frac{\lambda_s^2 s}{\varkappa^2}+ \lambda_o^2o\big)^{1/2}$ and
$x=\big\|[\bfSigma^{1/2}\bfDelta^{\bbeta} \,;\, \bfDelta^{\btheta}]\big\|_2$.
These notations are valid in this proof only. From \eqref{eq:27}, \eqref{eq:28},
\eqref{eq:29}, we get
\begin{align}
		\sc_1 x \le A x + \sqrt{Bx}\quad\Longrightarrow
		\quad x \le \frac{B}{(\sc_1 -A)^2}
\end{align}
provided that ${A\le \sc_1}$. Assuming ${2A\le \sc_1}$, we get
\begin{align}
	\big\|\bfSigma^{1/2}\bfDelta^{\bbeta}\|_2^2 + \|\bfDelta^{\btheta}\big\|_2^2
	\le \frac{16B^2}{\sc_1^4}.
\end{align}
For deriving the bound on the $\ell_1$ norms of the errors, we first use the
fact that $\bfDelta$ lies in the dimension reduction cone, followed by the
Cauchy-Schwarz inequality, to get
\begin{align}
	\lambda_s\big\|\bfDelta^{\bbeta}\|_1 + \lambda_o \|\bfDelta^{\btheta}\big\|_1
	&\le 4(\lambda_s\big\|\bfDelta^{\bbeta}_S\|_1 + \lambda_o \|\bfDelta^{\btheta}_O\big\|_1)\\
	&\le 4\bigg(\frac{\lambda_s^2 s}{\varkappa^2}+ \lambda_o^2o\bigg)^{1/2}
	\big\|[\bfSigma^{1/2}\bfDelta^{\bbeta} \,;\, \bfDelta^{\btheta}]\big\|_2\\
	&\le \frac{16B}{\sc_1^2}\bigg(\frac{\lambda_s^2 s}{\varkappa^2}+ \lambda_o^2o\bigg)^{1/2}\\
	& = \frac{24}{\sc_1^2}\bigg(\frac{\lambda_s^2 s}{\varkappa^2}+ \lambda_o^2o\bigg).
\end{align}

\item[Case 2:]$5\|\bfDelta^{\bbeta}_S\|_1 < \|\bfDelta^{\bbeta}_{S^c}\|_1$.
In this case, we can infer from the already proved fact $\bfDelta\in\mathcal
C_{S,O}(3,\gamma)$ that
\begin{align}\label{eq:30}
2\gamma\|\bfDelta^{\bbeta}_{S}\|_1+
\|\bfDelta^{\btheta}_{O^c}\|_1\le3\|\bfDelta^{\btheta}_{O}\|_1.
\end{align}
Hence, we have
\begin{align}
\sc_2\|\bfDelta^{\bbeta}\|_1 + \sc_3\|\bfDelta^{\btheta}\|_1 &\le
\Big(\frac{\sc_2}{\lambda_s}\bigvee \frac{\sc_3}{\lambda_o}\Big)\big(
\lambda_s\|\bfDelta^{\bbeta}\|_1 + \lambda_o\|\bfDelta^{\btheta}\|_1)\\
&\le 4
\Big(\frac{\sc_2}{\lambda_s}\bigvee \frac{\sc_3}{\lambda_o}\Big)\big(
\lambda_s\|\bfDelta^{\bbeta}_S\|_1 + \lambda_o\|\bfDelta^{\btheta}_O\|_1)\\
&\le 10
\Big(\frac{\sc_2}{\lambda_s}\bigvee \frac{\sc_3}{\lambda_o}\Big)
\lambda_o\|\bfDelta^{\btheta}_O\|_1\\
&\le 10
\Big(\frac{\sc_2}{\lambda_s}\bigvee \frac{\sc_3}{\lambda_o}\Big)
\lambda_o \sqrt{o}\,\|\bfDelta^{\btheta}\|_2.\label{eq:28'}
\end{align}
Similarly, the right hand side of \eqref{eq:27} can be bounded by the square-root
of the expression
\begin{align}
	3(\nicefrac{\lambda_s}{2})\|\bfDelta^{\bbeta}_S\|_1 +
    3(\nicefrac{\lambda_o}{2})\|\bfDelta^{\btheta}_O\|_1
	&\le (15/4)\lambda_o\|\bfDelta^{\btheta}_O\|_1
	\le (15/4)\lambda_o \sqrt{o}\,\|\bfDelta^{\btheta}\|_2.\label{eq:29'}
\end{align}
To ease notation, we define $A' =10
\big(\frac{\sc_2}{\lambda_s}\bigvee \frac{\sc_3}{\lambda_o}\big)
\lambda_o\sqrt{o}$, $B' = (15/4)\lambda_o\sqrt{o}$ and
$x'=\big\|[\bfSigma^{1/2}\bfDelta^{\bbeta}\,;\,\bfDelta^{\btheta}]\big\|_2$.
These notations are valid in this proof only. From \eqref{eq:27}, \eqref{eq:28'},
\eqref{eq:29'}, we get
\begin{align}
		\sc_1 x' \le A' x' + \sqrt{B'x'}\quad\Longrightarrow
		\quad x' \le \frac{B'}{(\sc_1 -A')^2} \le \frac{4B'}{\sc_1^2}
\end{align}
provided that $2A'\le \sc_1$. Thus, we have proved the inequality
\begin{align}
		\big\|\bfSigma^{1/2}\bfDelta^{\bbeta} \big\|_2\vee
		\big\|\bfDelta^{\btheta}\big\|_2 \le\frac{15\lambda_o\sqrt{o}}{\sc_1^2},
\end{align}
which implies that
\begin{align}
		\gamma \big\|\bfDelta^{\bbeta}\big\|_1 +\big\|\bfDelta^{\btheta}\big\|_1
		&\le 4(\gamma \big\|\bfDelta^{\bbeta}_{S}\big\|_1 +\big\|\bfDelta^{\btheta}_O\big\|_1 )
		\le 10\big\|\bfDelta^{\btheta}_O\big\|_1
		\le 10\sqrt{o}\,\big\|\bfDelta^{\btheta}_O\big\|_2\le
		\frac{150\lambda_o o}{\sc_1^2}.
\end{align}
\end{description}
To complete the proof, it suffices to remark that the upper bounds
provided in the statement of the proposition are larger than the bounds
we have just established both in case 1 and in case 2.
\end{proof}

\subsection{Augmented transfer principle and incoherence imply the nearly optimal rate}

\begin{lemma}\label{lemma:recursion:delta:bb}
The following bound holds:
\begin{align}
\Vert\bfX^{(n)}\bfDelta^{\bbeta}\Vert_2^2\le -(\bfDelta^{\bbeta})^\top
(\bfX^{(n)})^\top\bfDelta^{\btheta} +\Vert\bfDelta^{\bbeta}\Vert_1
\Vert(\bfX^{(n)})^\top\bxi^{(n)}\Vert_\infty
+\lambda_s\left(2\Vert\bfDelta^{\bbeta}_S\Vert_1-\Vert\bfDelta^{\bbeta}\Vert_1\right).
\end{align}
\end{lemma}
\begin{proof}
We note that
\begin{align}
\hat\bbeta\in\argmin_{\bbeta}\left\{\frac{1}{2}\left\Vert\by^{(n)}-\bfX^{(n)}\bbeta - \hat\btheta\right\Vert^2_2+\lambda_s\Vert\bbeta\Vert_1\right\}.
\end{align}
The KKT conditions of the above minimization problem imply that, for some $\bv\in\re^p$ such that $\Vert\bv\Vert_\infty\le1$ and  $\bv^\top\hat\bbeta=\Vert\hat\bbeta\Vert_1$,
\begin{align}
\mathbf 0&=(\bfX^{(n)})^\top\left(\bfX^{(n)}\hat\bbeta + \hat\btheta-
\by^{(n)}\right)+\lambda_s\bv\\
&=(\bfX^{(n)})^\top\left(\bfX^{(n)}\bfDelta^{\bbeta} + \bfDelta^{\btheta}-\bxi^{(n)}\right)+\lambda_s\bv.
\end{align}
Multiplying the above equality from the left by
$(\bfDelta^{\bbeta})^\top$ we obtain
\begin{align}
0&=\Vert\bfX^{(n)}\bfDelta^{\bbeta}\Vert_2^2 +
(\bfDelta^{\bbeta})^\top(\bfX^{(n)})^\top\bfDelta^{\btheta}
-(\bfDelta^{\bbeta})^\top(\bfX^{(n)})^\top\bxi^{(n)}+
\lambda_s(\hat\bbeta-\bbeta^*)^\top\bv.
\end{align}
From the above inequality, $\bv^\top\hat\bbeta=\Vert\bbeta\Vert_1$ and the fact that
$\bv^\top\bbeta^*\le\Vert\bbeta^*\Vert_1$ (since $\Vert\bv\Vert_\infty\le1$), we obtain that
\begin{align}
\Vert\bfX^{(n)}\bfDelta^{\bbeta}\Vert_2^2 \le -
(\bfDelta^{\bbeta})^\top(\bfX^{(n)})^\top\bfDelta^{\btheta}
 + \Vert\bfDelta^{\bbeta}\Vert_1\Vert(\bfX^{(n)})^\top\bxi^{(n)}\Vert_\infty
 + \lambda_s\big(\Vert\bbeta^*\Vert_1-\Vert\hat\bbeta\Vert_1\big).
\end{align}
One checks that
\begin{align}
\Vert\bbeta^*\Vert_1-\Vert\hat\bbeta\Vert_1\le\Vert\bfDelta^{\bbeta}_S\Vert_1-\Vert\bfDelta^{\bbeta}_{S^c}
\Vert_1=2\Vert\bfDelta^{\bbeta}_S\Vert_1-\Vert\bfDelta^{\bbeta}\Vert_1.
\end{align}
Combining this and the previous inequality we get the claim of the lemma.
\end{proof}

\begin{proposition}\label{prop:improved:rate}
Let $\bfSigma$ satisfy the $\RE(s,5)$ with constant $\varkappa>0$.
Let $\sa_1$, $\sa_2$, $\sa_3$, $\sb_1$, $\sb_2$, $\sc_1$, $\sc_2$,
$\sc_3$ and $\gamma$ be some positive real numbers
satisfying
\begin{align}
		8\big({\sc_2}\vee \gamma{\sc_3}\big)
		\bigg(\frac{s}{\varkappa^2}+ \frac{6.25 o}{\gamma^2}\bigg)^{1/2}
		&\le \sc_1\label{cond1}\\		
		{36\sb_2}\bigg(\frac{s}{\varkappa^2}+ \frac{6.25 o}{\gamma^2}\bigg)^{1/2}
		&\le {\sc_1^2}  \label{cond2}.
\end{align}
Assume that on some event $\Omega$, the following conditions are met:
\begin{itemize}\itemsep=0pt
\item[\rm(i)] $\bfX$ satisfies the $\TP\left(\sa_1;\sa_2\right)$.
\item[\rm(ii)] $\bfX$ satisfies the $\IP\left(\sb_1;\sb_2;\sb_3\right)$ .
\item[\rm(iii)] $\bfX$ satisfies the $\ATP\left(\sc_1;\sc_2;\sc_3\right)$ .
\item[\rm(iv)]
$
\lambda_s= \gamma\lambda_o \ge
(\nicefrac{2}{n})\Vert\bfX^\top\bxi\Vert_\infty,\quad\text{and}\quad
\lambda_o \ge (\nicefrac{2}{\sqrt{n}})\Vert\bxi\Vert_\infty
$.
\end{itemize}
Then, on the same event $\Omega$, we have
\begin{align}
		\big\|{\bfSigma}^{1/2}(\hat\bbeta-\bbeta^*)\big\|_2 \le
		\frac{48\lambda_s}{\sc_1^2}\bigg(\frac{2\sa_2}{\sa_1}\bigvee
		\frac{(\sb_1 +\sb_3)\gamma}{\sa_1^2}\bigg)
		\bigg(\frac{s}{\varkappa^2}+ \frac{6.25o}{\gamma^2}\bigg)
		+ \frac{5\lambda_s\sqrt{s}}{3\sa_1^2\varkappa}.
\label{prop:improved:rate:sample:size}
\end{align}
\end{proposition}
\begin{proof}
Assume that we the event $\Omega$ is realized.
Condition \eqref{cond1} implies that the claims of \Cref{prop:augmented:analysis:v2}
hold true. In particular, the Euclidean norm of the error of estimating $\btheta^*$
can be bounded as follows:
\begin{align}
\|\bfDelta^{\btheta}\|_2
		\le \frac{6}{\sc_1^2} \bigg(\frac{\lambda_s^2 s}{\varkappa^2}+
		6.25\lambda_o^2o\bigg)^{1/2}\le \frac{\lambda_s}{6\sb_2},\label{cond3}
\end{align}
where the last inequality follows from \eqref{cond2}.
\Cref{lemma:recursion:delta:bb} and item (ii) imply that
\begin{align}
\Vert\bfX^{(n)}\bfDelta^{\bbeta}\Vert_2^2
&\le (\bfDelta^{\bbeta})^\top(\bfX^{(n)})^\top\bfDelta^{\btheta}
+\Vert\bfDelta^{\bbeta}\Vert_1\Vert(\bfX^{(n)})^\top\bxi^{(n)}\Vert_\infty
+\lambda_s\left(2\Vert\bfDelta^{\bbeta}_S\Vert_1-\Vert\bfDelta^{\bbeta}\Vert_1\right)\\
&\stackrel{(iv)}{\le} (\bfDelta^{\bbeta})^\top(\bfX^{(n)})^\top\bfDelta^{\btheta}
+\frac{\lambda_s}{2}\Vert\bfDelta^{\bbeta}
\Vert_1
+\lambda_s\left(2\Vert\bfDelta^{\bbeta}_S\Vert_1-\Vert\bfDelta^{\bbeta}\Vert_1\right)\\
&\stackrel{\IP}{\le}
\sb_1\big\|{\bfSigma}^{1/2}\bfDelta^{\bbeta}\big\|_2
\Vert\bfDelta^{\btheta}\Vert_2+\sb_3
\big\Vert{\bfSigma}^{1/2}\bfDelta^{\bbeta}\big\Vert_2\Vert\bfDelta^{\btheta}\Vert_1+
2\lambda_s\Vert\bfDelta^{\bbeta}_S\Vert_1-\frac{\lambda_s}{3}\Vert\bfDelta^{\bbeta}\Vert_1\\
&\quad+\sb_2\Vert\bfDelta^{\bbeta}\Vert_1\Vert\bfDelta^{\btheta}\Vert_2 - \frac{\lambda_s}{6}\Vert\bfDelta^{\bbeta}\Vert_1\\
&\stackrel{}{\le}
\sb_1\big\|{\bfSigma}^{1/2}\bfDelta^{\bbeta}\big\|_2
\Vert\bfDelta^{\btheta}\Vert_2+\sb_3
\big\Vert{\bfSigma}^{1/2}\bfDelta^{\bbeta}\big\Vert_2\Vert\bfDelta^{\btheta}\Vert_1+
(\lambda_s/3)\big(5\Vert\bfDelta^{\bbeta}_S\Vert_1 - \Vert\bfDelta^{\bbeta}_{S^c}\Vert_1\big)
\label{prop:improved:rate:eq1}
\end{align}
where the last line follows from the fact that $2\Vert\bfDelta^{\bbeta}_S\Vert_1-
\nicefrac{1}{3}\Vert\bfDelta^{\bbeta}\Vert_1 = \nicefrac{1}{3}
(5\Vert\bfDelta^{\bbeta}_S\Vert_1-\Vert\bfDelta^{\bbeta}_{S^c}\Vert_1)$
and \eqref{cond3}. To ease notation, let us use notations
$A =\sb_1 \Vert\bfDelta^{\btheta}\Vert_2+\sb_3
\Vert\bfDelta^{\btheta}\Vert_1$, $B = \nicefrac{\lambda_s}{3}
\big(5\Vert\bfDelta^{\bbeta}_S\Vert_1 - \Vert\bfDelta^{\bbeta}_{S^c}\Vert_1\big)_+$
and $x = \big\|{\bfSigma}^{1/2}\bfDelta^{\bbeta}\big\|_2$, which are valid for
this proof only. On the one hand, combining the last inequality and the $\TP$, we
arrive at
\begin{align}
    (\sa_1 x -\sa_2\|\bfDelta^{\bbeta}\|_1)_+^2\le A x+ B.
\end{align}
This implies that either $x\le (\sa_2/\sa_1) \|\bfDelta^{\bbeta}\|_1$ or
\begin{align}
		\Big(\sa_1 x -\sa_2\|\bfDelta^{\bbeta}\|_1-\frac{A}{2\sa_1}\Big)^2\le
		B + \frac{A^2}{4\sa_1^2} + \frac{A\sa_2 }{\sa_1}\,\|\bfDelta^{\bbeta}\|_1.
\end{align}
Therefore, in both cases,
\begin{align}
		x &\le  \frac{\sa_2}{\sa_1}\|\bfDelta^{\bbeta}\|_1 + \frac{A}{2\sa_1^2} +
		\frac1{\sa_1}\Big\{B + \frac{A^2}{4\sa_1^2} + \frac{A\sa_2 }{\sa_1}\,
		\|\bfDelta^{\bbeta}\|_1\Big\}^{1/2}
		\le  \frac{2\sa_2}{\sa_1}\|\bfDelta^{\bbeta}\|_1 + \frac{A}{\sa_1^2} +
		\frac{B^{1/2}}{\sa_1}.\label{ineq:22}
\end{align}
On the other hand, the $\RE(s,5)$ property yields
\begin{align}
	B\le \frac{5\lambda_s\Vert\bfDelta^{\bbeta}_S\Vert_1}{3}
	\le \frac{5\lambda_s\sqrt{s}\Vert\bfDelta^{\bbeta}_S\Vert_2}{3}
	\le \frac{5\lambda_s\sqrt{s}\,x}{3\varkappa}\le \bigg(\frac{\sa_1 x}{2} +
	\frac{5\lambda_s\sqrt{s}}{6\sa_1\varkappa}\bigg)^2.
	\label{ineq:23}
\end{align}
Combining \eqref{ineq:22} and \eqref{ineq:23}, we get
\begin{align}
		\frac{x}{2}
		\le  \frac{2\sa_2}{\sa_1}\|\bfDelta^{\bbeta}\|_1 + \frac{A}{\sa_1^2} +
	 \frac{5\lambda_s\sqrt{s}}{6\sa_1^2\varkappa}.\label{ineq:24}
\end{align}
Replacing $A$ and $x$ by their expressions, we arrive at
\begin{align}
		\frac{1}{2}\big\|{\bfSigma}^{1/2}\bfDelta^{\bbeta}\big\|_2 & \le
		\frac{2\sa_2}{\sa_1}\|\bfDelta^{\bbeta}\|_1 +
		\frac{\sb_1 \Vert\bfDelta^{\btheta}\Vert_2+\sb_3\Vert\bfDelta^{\btheta}\Vert_1}{\sa_1^2}
		+ \frac{5\lambda_s\sqrt{s}}{6\sa_1^2\varkappa}\\
		&\le
		\frac{2\sa_2}{\sa_1}\|\bfDelta^{\bbeta}\|_1 +
		\frac{\sb_1 +\sb_3}{\sa_1^2} \,\Vert\bfDelta^{\btheta}\Vert_1
		+ \frac{5\lambda_s\sqrt{s}}{6\sa_1^2\varkappa}\\
		&\le
		\bigg(\frac{2\sa_2}{\gamma\sa_1}\bigvee \frac{\sb_1 +\sb_3}{\sa_1^2}\bigg)
		\big(\gamma\|\bfDelta^{\bbeta}\|_1 + \Vert\bfDelta^{\btheta}\Vert_1\big)
		+ \frac{5\lambda_s\sqrt{s}}{6\sa_1^2\varkappa}.\label{ineq:25}
\end{align}
Finally, combining inequality \eqref{equation:aux:rate:l1:v2} from
\Cref{prop:augmented:analysis:v2} with the last display  we obtain
\begin{align}
		\frac{1}{2}\big\|{\bfSigma}^{1/2}\bfDelta^{\bbeta}\big\|_2 &\le
		\frac{24\lambda_o}{\sc_1^2}\bigg(\frac{2\sa_2}{\gamma\sa_1}\bigvee
		\frac{\sb_1 +\sb_3}{\sa_1^2}\bigg)
		\bigg(\frac{\gamma^2 s}{\varkappa^2}+ 6.25o\bigg)
		+ \frac{5\lambda_s\sqrt{s}}{6\sa_1^2\varkappa}.\label{ineq:26}
\end{align}
This completes the proof of the proposition.
\end{proof}

\section{Properties of Gaussian matrices}\label{sec:pr2}

The next lemma ensures that the parameters $\lambda$ and $\gamma$ satisfy,
with high-probability, condition ii) of \Cref{prop:augmented:analysis:v2} 
(which is the same as (iv) of \Cref{prop:improved:rate}).

\begin{lemma}\label{lemma:penalization:factors}
Let the rows of $\bfZ$ be iid Gaussian with zero mean and covariance
matrix $\bfSigma$ and $\bxi\sim\mathcal N_n(\mathbf 0,\sigma^2\bfI_n)$.
Then the following two claims hold true.
\begin{itemize}
\item[(i)] For any $\delta\in(0,1]$, with probability at least $1-\delta$,
\begin{align}
\max_{j\in[p]}\Vert\bfZ^{(n)}_{\bullet,j}\Vert_2\le
\bigg\{1+\sqrt{\frac{2\log(p/\delta)}{n}}\bigg\}\rho(\bfSigma).
\end{align}
\item[(ii)] For any $\delta\in(0,1]$ and $n\ge2\log(3p/\delta)$,
penalization factors such that
\begin{align}
\lambda_o\ge 2\sigma\sqrt{\frac{2\log(3n/\delta)}{n}},\quad\quad
\lambda_s\ge 2\sigma\rho(\bfSigma)\sqrt{\frac{2\log(3p/\delta)}{n}}
\bigg(1+\sqrt{\frac{2\log(3p/\delta)}{n}}\bigg),
\end{align}
satisfy conditions of item (iv) of \Cref{prop:improved:rate}
with probability at least $1-\delta$.
\end{itemize}
\end{lemma}
\begin{proof}
Let $\widetilde\bfZ:=\bfZ\bfSigma^{-{1}/{2}}$. We also note that
\begin{align}
\Vert\bfZ_{\bullet,j}\Vert_2^2=\sum_{i\in[n]}\left[\widetilde\bfZ_{i,\bullet}(\bfSigma^{{1}/{2}})_{\bullet,j}\right]^2,
\end{align}
where $\widetilde\bfZ_{1,\bullet}(\bfSigma^{1/2})_{\bullet,j},\ldots,\widetilde\bfZ_{n,\bullet}(\bfSigma^{1/2})_{\bullet,j}$ are iid $\mathcal{N}(0,\bfSigma_{jj})$. By standard $\chi^2$ concentration inequalities, for all $j\in[p]$, with probability at least $1-\delta/p$,
\begin{align}
\Vert\bfZ_{\bullet,j}^{(n)}\Vert_2\le\bfSigma_{jj}^{1/2}
\bigg\{1+\sqrt{\frac{2\log(p/\delta)}{n}}\bigg\}.
\end{align}
Item (i) follows from this inequality using the union bound.

We now prove item (ii). Recall that $\bfZ$ and $\bxi\sim\mathcal{N}_n(0,\sigma^2\bfI_n)$ are independent and, therefore, conditionally on $\bfZ$, $(\bfZ_{\bullet,j})^\top\bxi\sim\mathcal{N}_n(0,\sigma^2\Vert\bfZ_{\bullet,j}\Vert^2_2)$. The well known maximal Gaussian concentration inequality implies that for all $j\in[p]$, with probability at least $1-\delta/3p$,
\begin{align}
|(\bfZ^{(n)}_{\bullet,j})^\top\bxi^{(n)}|\le\sigma\Vert\bfZ_{\bullet,j}^{(n)}\Vert2\sqrt{\frac{2\log(3p/\delta)}{n}}.\label{lemma:penalization:factors:eq2}
\end{align}
Similarly, with probability at least $1-\delta/3$,
\begin{align}
\Vert\bxi^{(n)}\Vert_\infty\le\sigma\sqrt{\frac{2\log(3n/\delta)}{n}}. \label{lemma:penalization:factors:eq3}
\end{align}
Taking the union bound over the $p$ sets satisfying \eqref{lemma:penalization:factors:eq2}, the set satisfying \eqref{lemma:penalization:factors:eq3} and the set satisfying item (i), we prove item (ii).
\end{proof}

\subsection{Bounding extrema on compact sets}\label{ss9.1}

In what follows, we will use the notion of Gaussian width for measuring
the richness of a set of vectors. For a compact set $\mathcal{B}\subset
\mathbb{R}^{p}$, we define the Gaussian width of $\mathcal{B}$
by
\begin{align}
\mathscr G(\mathcal B) := \esp\Big[\sup_{\bb\in \mathcal B} \bb^\top \bxi\Big],
\qquad \bxi_{i}\stackrel{\rm iid}{\sim} \mathcal N(0,1).
\end{align}
In view of   \citep[Theorem 2.5]{boucheron:lugosi:massart2013}, for
every symmetric $p\times p$ matrix $\bfA$,
$\esp\left[\Vert \mathbf{A}\bxi\Vert_\infty\right]\le \{\max_{j\in[p]}(\bfA^2)_{jj}^{1/2}\}
\sqrt{2\log p}$. This implies that
\begin{align}\label{calG}
    \mathscr{G}(\mathbf{A}\mathbb{B}^p_1)  = \esp[\|\mathbf{A}\bxi\|_{\infty}] \le
    \rho(\bfA^2)\,\sqrt{2\log p}\, .
\end{align}
The above inequality is tight for orthogonal matrices $\mathbf{A}$,
but it might be sub-optimal, up to a log factor, especially for
poorly conditioned matrices $\mathbf{A}$.

\begin{lemma}\label{lemma:aux1}
Let $\bfZ$ be a $n\times p$ matrix with iid $\mathcal N(0,1)$ entries. For all $n\ge1$, $t>0$ and any compact set $\mathcal B\subset\mathbb{S}^{p-1}$, with probability at least $1-\exp(-t^2/2)$,
$$
\inf_{\bb\in \mathcal B}\big\|\bfZ\bb\big\|_2\ge \frac{n}{\sqrt{n+1}}-
\mathscr G(\mathcal B)-t.
$$
As a consequence, for all $n\ge1$ and $\delta\in(0,1]$, with with probability at least $1-\delta$, the following inequality holds:
$$
\inf_{\bb\in \mathcal B}\big\Vert\bfZ^{(n)}\bb\big\Vert_2\ge 1-\frac{1}{2n}-\sqrt{\frac{2\log(1/\delta)}{n}}-\frac{\mathscr G(\mathcal B)}{\sqrt{n}}.
$$
\end{lemma}

\begin{proof}
The norm of $\bfZ\bb$ can be written as
\begin{align}
\|\bfZ\bb\|_2 = \sup_{\bv\in\mathbb{B}_2^n} \bv^\top \bfZ\bb.
\end{align}
We define the centered Gaussian process $Z_{\bb,\bv} = -\bv^\top\bfZ\bb
= -\sum_{i=1}^n  \bfZ_i\bb\bv_i$. It satisfies
\begin{align}
\esp[(Z_{\bb,\bv} -Z_{\bb',\bv'})^2] = \|\bb\bv^\top-\bb'(\bv')^\top\|_F^2.
\end{align}
We are interested in upper bounding the quantity $\inf_{\bv}\sup_{\bb} Z_{\bb,\bv}$.
To this end, we define the process
$$
W_{\bb,\bv} = \bv^\top\bxi + \bb^\top\bar\bxi,
$$
where $\bxi\in\mathbb{R}^{n}$ and $\bar\bxi\in\mathbb{R}^{p}$ are two
independent vectors with iid $\mathcal N(0,1)$ entries. One checks that
\begin{align}
\esp[(Z_{\bb,\bv} -Z_{\bb',\bv'})^2] -
\esp[(W_{\bb,\bv} -W_{\bb',\bv'})^2]
&= \|\bb\bv^\top-\bb'(\bv')^\top\|_F^2 - \|\bv-\bv'\|_F^2-
\|\bb-\bb'\|_F^2\\
& = -2(1-\bv^\top\bv')(1-\bb^\top\bb')\le 0.
\end{align}
Using Gordon's inequality, we get
\begin{align}
\esp[\inf_{\bv}\sup_{\bb} Z_{\bb,\bv}] \le
\esp[\inf_{\bv}\sup_{\bb} W_{\bb,\bv}]  = \mathscr G(\mathcal B)
-\esp[\|\bxi\|_2]\le \mathscr G(\mathcal B)
-\frac{n}{\sqrt{n+1}}.
\end{align}
To complete the proof of the first statement, it suffices to note that the
mapping $\bfZ\mapsto\inf_{\bb\in \mathcal B}\left\Vert\bfZ\bb\right\Vert_2$
is Lipschitz with constant 1, and to apply the Gaussian concentration
inequality \citep[Theorem 5.6]{boucheron:lugosi:massart2013}.
Scaling the obtained bound by ${1}/{\sqrt{n}}$, the proof of the inequality
in the second statement is immediate after we use the simple bound
$(\nicefrac{{n}}{{n+1}})^{1/2}\ge 1-\nicefrac{1}{2n}$.
\end{proof}

\begin{lemma}\label{lem:GW1}
Let $\bfZ$ be a $n\times p$ matrix with iid $\mathcal N(0,1)$ entries.
Let $V$ be any compact subset of $\mathbb{S}^{p-1}\times \mathbb{S}^{n-1}$
and define $V_1= \{\bv:\exists\, \bu \text{ s.t. } (\bv,\bu)\in V\}$
and $V_2 = \{\bu:\exists\, \bv \text{ s.t. } (\bv,\bu)\in V\}$.
Then for any $n\ge1$ and $t>0$, with probability at least $1-\exp(-t^2/2)$,
we have
$$
\sup_{[\bv;\bu]\in V}\bu^\top\bfZ\bv \le \mathscr G\big(V_1) + \mathscr G\big(V_2\big) + t.
$$
\end{lemma}

\begin{proof}
For each $(\bv,\bu)\in V$, we define
\begin{align}
Z_{\bv,\bu}&:=
\bu^\top\bfZ\bv,\qquad
W_{\bv,\bu}:= \bv^\top\bxi+\bu^\top\bar{\bxi},
\end{align}
where $\bxi$ and $\bar\bxi$ are two independent standard Gaussian vectors. Therefore,
$(\bv,\bu)\mapsto Z_{\bv,\bu}$ and $(\bv,\bu)\mapsto W_{\bv,\bu}$ define
centered continuous Gaussian processes $W$ and $Z$ indexed by $ V$.

To compute the variance of the increments of $W$. We remark that
\begin{align}
Z_{\bv,\bu}-Z_{\bv',\bu'}=\tr[\bfZ(\bv\bu^\top-\bv'(\bu')^\top)]
\sim {\calN}(0,\Vert \bv\bu^\top-\bv'(\bu'^\top)\Vert_F^2).
\end{align}
Hence,
\begin{align}
\esp\big[\big(Z_{\bv,\bu}-Z_{\bv',\bu'}\big)^2\big]&=\Vert \bv\bu^\top-
\bv'(\bu')^\top\Vert_{F}^2 =\Vert (\bv-\bv')\bu^\top+\bv'(\bu-\bu')^\top
\Vert_{F}^2 \\
&\le \Vert\bv- \bv'\Vert_2^2+\Vert \bu-\bu'\Vert_2^2,
\label{lemma:aux2:eq1}
\end{align}
using Cauchy-Schwarz's inequality  and the facts that $\bv,\bv'\in\mathbb{S}^{p-1}$
and $\bu,\bu'\in\mathbb{S}^{n-1}$. On the other hand, the definition of the
process $Z$ yields
\begin{align}
\esp[(W_{\bv,\bu}-W_{\bv',\bu'})^2]=\Vert \bv-\bv'\Vert_2^2+\Vert \bu-\bu'\Vert_2^2.
\label{lemma:aux2:eq2}
\end{align}

From \eqref{lemma:aux2:eq1},\eqref{lemma:aux2:eq2}, we conclude that the centered
Gaussian processes $W$ and $Z$ satisfy the conditions of Gordon's inequality. Hence,
using the notation $V_1= \{\bv:\exists\, \bu \text{ s.t. } (\bv,\bu)\in V\}$
and $V_2 = \{\bu:\exists\, \bv \text{ s.t. } (\bv,\bu)\in V\}$, we get
\begin{align}
\esp\bigg[\sup_{[\bv;\bu]\in V} Z_{\bv,\bu}\bigg]&\le \esp\bigg[\sup_{[\bv;\bu]\in V}
W_{\bv,\bu}\bigg]\le\esp\bigg[\sup_{\bv\in V_1}\bv^\top\bxi\bigg]+
\esp\bigg[\sup_{\bu\in V_2}\bu^\top\bar\bxi\bigg]
=\mathscr{G}(V_1) + \mathscr{G}(V_2).
\end{align}
Moreover, $\bfZ\mapsto\sup_{[\bv;\bu]\in V_1\times V_2}\bu^\top\bfZ\bv$
is Lipschitz continuous with constant 1, so the Gaussian concentration
inequality holds \citep[Theorem 5.6]{boucheron:lugosi:massart2013}. This and the previous
inequality bounding the mean complete the proof.
\end{proof}

\subsection{Removing compactness constraints: peeling techniques}\label{ss:peeling}

\begin{lemma}[Single-parameter peeling]\label{lemma:peeling:1dim}
Let $g:\mathbb{R}_+\to\mathbb{R}_+$ be a right-continuous non-decreasing
function  and  $h:V\to\mathbb{R}_+$. Assume that for some  constants
$b\in\mathbb{R}_+$ and $c\ge 1$, for every $r>0$ and for any $\delta\in(0,1/(7\vee c))$,
we have
\begin{align}
A(r,\delta) = \Big\{\inf_{\bv\in V: h(\bv)\le r} M(\bv)
\ge - g(r)-b\sqrt{\log(1/\delta)}\Big\},
\label{lemma:expectation:eq5}
\end{align}
with probability at least $1-c\delta$. Then, with probability at least
$1-c\delta$, we have
\begin{align}
    \forall \bv\in V\quad M(\bv)
    \ge -1.2(g\circ h)(\bv) - \big(3+\sqrt{\log(9/\delta)}\big)b.
\end{align}
\end{lemma}
\begin{proof} Throughout the proof, without loss of generality, we assume
$b=1$. Let $\eta,\epsilon>1$ be two parameters to be chosen later on.  We
set\footnote{Here $g^{-1}$ is the generalized inverse
defined by $g^{-1}(x) = \inf\{a\in\mathbb{R}_+: g(a)\ge x\}$. }
$\mu_0=0$, $\mu_k = \mu \eta^{k-1}$, $\nu_k  =g^{-1}(\mu_k)$ and
$V_k =\{\bv\in V : \mu_k\le (g\circ h)(\bv)< \mu_{k+1}\}$, for $k\ge 1$.
The union bound and the fact that $\sum_{k\ge 1} k^{-1-\epsilon}\le
1+\epsilon^{-1}$ imply that the event
$$
A := \bigcap_{k=1}^\infty A(\nu_k, \epsilon\delta/((1+\epsilon)k^{1+\epsilon}))
$$
has a probability at least $1-c\delta$. We assume in the sequel that this event is realized, that is
\begin{align}
		\forall k\in\mathbb{N}^*
		\quad
		\begin{cases}
    \forall\bv\in V \text{ such that } h(\bv)\le \nu_k\text{ we have } \\
		M(\bv) \ge - g(\nu_k)-\sqrt{\log\{(1+\epsilon)/(\epsilon\delta)\}
		+ ({1+\epsilon})\log k},\qquad.
		\end{cases}\label{eq:16}
\end{align}
For every $\bv\in V$, there is $\ell\in\mathbb{N}$ such that
$\bv\in V_\ell$. If $\ell\ge 1$, then $h(\bv)\le \nu_{\ell+1}$ and \eqref{eq:16} implies that
\begin{align}
    M(\bv) &\ge -g(\nu_{\ell+1}) - \sqrt{\log\{(1+\epsilon)/(\epsilon\delta)\}
		+ ({1+\epsilon})\log (\ell+1)} \\
    &= - \mu_{\ell+1} - \sqrt{\log\{(1+\epsilon)/(\epsilon\delta)\}
		+ ({1+\epsilon})\log (\ell+1)} \\
		&= -\eta \mu_{\ell} - \sqrt{\log\{(1+\epsilon)/(\epsilon\delta)\}
		+ ({1+\epsilon})\log (\ell+1)}\\
    &\ge -\eta^2  (g\circ h)(\bv) + (\eta-1)\mu \eta^{\ell} -
		\sqrt{\log\{(1+\epsilon)/(\epsilon\delta)\} + ({1+\epsilon})\log (\ell+1)}.\label{eq:19}
\end{align}
If $\ell=0$, then \eqref{eq:16} with $k=1$ leads to
\begin{align}
    M(\bv) &\ge -g(\nu_1) - \sqrt{\log\{(1+\epsilon)/(\epsilon\delta)\}} \\
		&= -g(g^{-1}(\mu)) - \sqrt{\log\{(1+\epsilon)/(\epsilon\delta)\}} \\		
		&= -\mu- \sqrt{\log\{(1+\epsilon)/(\epsilon\delta)\}}
		.\label{eq:20}
\end{align}
From \eqref{eq:19} one can infer that, for $\ell\ge 1$,
\begin{align}
    M(\bv)
    &\ge  -\eta^2  (g\circ h)(\bv) - \sqrt{\log\{(1+\epsilon)/(\epsilon\delta)\}}\\
		&\qquad+ \eta^{\ell}\Big((\eta-1)\mu - \sup_{z\ge 1}
		\frac{\sqrt{\log\{(1+\epsilon)/(\epsilon\delta)\}
		+ ({1+\epsilon})\log (z+1)}-\sqrt{\log\{(1+\epsilon)/(\epsilon\delta)\}}}{\eta^z}\Big).
\end{align}
We choose $\mu$ so that the last term vanishes, that is
\begin{align}
  {(\eta-1)\mu}
		&=\sup_{z\ge 1}\frac{\sqrt{\log\{(1+\epsilon)/(\epsilon\delta)\}
		+ ({1+\epsilon})\log (z+1)}-\sqrt{\log\{(1+\epsilon)/(\epsilon\delta)\}}}{\eta^z}\\
    &=\sup_{z\ge 1}\frac{({1+\epsilon})\eta^{-z}\log (z+1)}{
		\sqrt{\log\{(1+\epsilon)/(\epsilon\delta)\}
		+ ({1+\epsilon})\log (z+1)}+\sqrt{\log\{(1+\epsilon)/(\epsilon\delta)\}}}.
\end{align}
To compute the last expression, we choose $\eta^2=1.2$ and $\epsilon = 1/8$. This
yields
\begin{align}
  {\mu}
    &=(\eta-1)^{-1}\sup_{z\ge 1}\frac{(9/8)(1.2)^{-z/2}\log (z+1)}
		{\sqrt{\log(9/\delta) + (9/8)\log (z+1)}+\sqrt{\log(9/\delta)}}\\
		&\le (\eta-1)^{-1}\sup_{z\ge 1}\frac{(9/8)(1.2)^{-z/2}\log (z+1)}
		{\sqrt{\log 36 + (9/8)\log (z+1)}+\sqrt{\log 36}}\le 3.
\end{align}
Combining with \eqref{eq:20}, this yields
\begin{align}
    M(\bv)
    &\ge -\mu -1.2(g\circ h)(\bv) - \sqrt{\log(9/\delta)}\\
    &\ge  -1.2(g\circ h)(\bv) - \big(3+\sqrt{\log(9/\delta)}\big).
\end{align}
This completes the proof.
\end{proof}

\begin{lemma}[Bi-parameter peeling]\label{lemma:peeling:2dim}
Let $g,\bar g$  be  right-continuous, non-decreasing functions
from $\mathbb{R}_+$ to $\mathbb{R}_+$ and  $h,\bar h$ be functions from
$V$ to $\mathbb{R}_+$.
Assume that for some  constants $b\in\mathbb{R}_+$ and $c\ge 1$, for every $r,\bar r>0$ and for any $\delta\in(0,1/(c\vee 7))$, we have
\begin{align}
A(r,\bar r,\delta) = \Big\{\inf_{\bv\in V: (h,\bar h)(\bv)\le (r,\bar r)}
M(\bv) \ge - g(r) -  \bar g(\bar r)-b\sqrt{\log(1/\delta)}\Big\},
\label{lemma:expectation:eq5}
\end{align}
with probability at least $1-c\delta$. Then,
with probability at least $1-c\delta$, we have
\begin{align}
    \forall \bv\in V\quad M(\bv)
    \ge -1.2(g\circ h)(\bv) -1.2(\bar g\circ\bar h)(\bv)-
		b\big(4.8+\sqrt{\log(81/\delta)}\big).
\end{align}
\end{lemma}
\begin{proof}
We will repeat the same steps as for the one-parameter peeling. W.l.o.g.\
we assume $b=1$. We choose $\mu>0$, $\eta>1$ and $\epsilon>0$.
Define \footnote{Here $g^{-1}$ is the generalized inverse given by
$g^{-1}(x) = \inf\{a\in\mathbb{R}_+: g(a)\ge x\}$. }
$\mu_0=0$, $\mu_k = \mu \eta^{k-1}$, $\nu_k  =g^{-1}(\mu_k)$,
$\bar\nu_k = \bar g^{-1}(\mu_k)$ and $V_{k,\bar k} =\{\bv\in V :
\mu_k\le (g\circ h)(\bv)< \mu_{k+1},\ \mu_{\bar k}\le
(\bar g\circ \bar h)(\bv)< \mu_{\bar k+1}\}$. The union bound
implies that  the event
$$
A = \bigcap_{k=1}^\infty A\Big(\nu_k,\bar\nu_{\bar k}, \frac{\epsilon^2
\delta}{(1+\epsilon)^2(k\bar k)^{1+\epsilon}}\Big)
$$
has a probability at least $1-c\delta$. To ease notation, set
$\delta_\epsilon = \epsilon^2\delta/(1+\epsilon)^2$.
We assume in the sequel that the event $A$ is realized, that is
\begin{align}
    \forall k,\bar k&\in\mathbb{N}^*,\ \forall\bv\in V \text{ such that }
		(h,\bar h)(\bv)\le (\nu_k,\bar\nu_{\bar k}) \text{ we have }\\
		M(\bv) &\ge - g(\nu_k) - \bar g(\bar\nu_{\bar k})-
		\sqrt{\log(1/\delta_\epsilon) + (1+\epsilon)\log (k\bar k)}.\label{eq:16bis}
\end{align}
For every $\bv\in V$, there is a pair $(\ell,\bar\ell)\in\mathbb{N}^2$
such that $\bv\in V_\ell$. If $\ell\wedge\bar\ell\ge 1$, then
$(h,\bar h)(\bv)\le (\nu_{\ell+1},\bar\nu_{\bar\ell+1})$, and
\eqref{eq:16bis} implies that
\begin{align}
    M(\bv) &\ge -g(\nu_{\ell+1}) - \bar g(\bar\nu_{\bar \ell+1})-
		\sqrt{\log(1/\delta_\epsilon) + (1+\epsilon)\log (\ell+1)(\bar \ell+1)}\\
    &= - \mu_{\ell+1} - \mu_{\bar\ell+1} -
		\sqrt{\log(1/\delta_\epsilon) + (1+\epsilon)\log (\ell+1)(\bar \ell+1)}\\
		&= -\eta \mu_{\ell} - \eta \mu_{\bar\ell} -
		\sqrt{\log(1/\delta_\epsilon) + (1+\epsilon)\log (\ell+1)(\bar \ell+1)}\label{eq:19bis}.
\end{align}
From this inequality, we infer that
\begin{align}
    M(\bv) &\ge  -\eta^2[(g\circ h)(\bv) + (\bar g\circ\bar h)(\bv)]\\
		&\quad + \eta(\eta-1)(\mu_\ell +\mu_{\bar\ell})
		- \sqrt{\log(1/\delta_\epsilon) + (1+\epsilon)\log (\ell+1)(\bar \ell+1)}\\
		& = -\eta^2(g\circ h)(\bv) -\eta^2(\bar g\circ\bar h)(\bv) -
			\sqrt{\log(1/\delta_\epsilon)}\\
		&\quad +\Big\{(\eta-1)\mu(\eta^\ell +\eta^{\bar\ell})+ \sqrt{\log(1/\delta_\epsilon)}
		- \sqrt{\log(1/\delta_\epsilon) + (1+\epsilon)\log (\ell+1)(\bar \ell+1)}\Big\}.
\end{align}
We choose $\mu$ so that the expression inside the braces is
nonnegative, that is
\begin{align}
    (\eta-1)\mu &=
		\sup_{z,\bar z\ge 1}\frac{\sqrt{\log(1/\delta_\epsilon)+ (1+\epsilon)\log(1+z)+
		(1+\epsilon)\log(1+\bar z)}-\sqrt{\log(1/\delta_\epsilon)}}{\eta^z+\eta^{\bar z}}.
\end{align}
Setting $\epsilon = 1/8$, $\eta^2 = 1.2$ and using that $\delta\le 1/7$, we get that
$\delta_\epsilon \le 1/567$ and hence
\begin{align}
    \mu &\le
		(\eta-1)^{-1}\sup_{z,\bar z\ge 1}\frac{\sqrt{\log 567+ (9/8)\log(1+z)+
		(9/8)\log(1+\bar z)}-\sqrt{\log 567}}{1.2^{z/2}+1.2^{\bar z/2}}\le 2.4
\end{align}
Combining with the case  $\ell\wedge\bar\ell=1$, this yields
\begin{align}
    M(\bv)
    &\ge -2\mu -1.2(g\circ h)(\bv) - 1.2(\bar g\circ\bar h)
		(\bv) - \sqrt{\log(81/\delta)}\\
    &\ge  -1.2(g\circ h)(\bv) - 1.2(\bar g\circ\bar h)
		(\bv) - 4.8-\sqrt{\log(81/\delta)}.
\end{align}
This completes the proof.
\end{proof}

\subsection{Structural properties of Gaussian designs}\label{ss9.3}

\begin{proposition}\label{pr:transfer:principle:1dim}
Let $\bfZ$ be a $n\times p$ matrix with iid $\mathcal N_p(0,\bfSigma)$ columns.
For all $n\ge100$ and $\delta\in(0,1/7]$, with  probability at least $1-\delta$,
the following inequality holds: for all $\bv\in\re^p$,
\begin{align}\label{TP:1}
\big\Vert\bfZ^{(n)}\bv\big\Vert_2\ge \Big(1-\frac{4.3+\sqrt{2\log(9/\delta)}}{\sqrt{n}}\Big)\Vert\bfSigma^{1/2}\bv\Vert_2-
\frac{1.2\mathscr G(\bfSigma^{1/2}\mathbb{B}_1^p)}{\sqrt{n}}
\Vert\bv\Vert_1.
\end{align}
\end{proposition}
\begin{remark}
The above result is similar to \cite[Theorem 1]{raskutti:wainwright:yu2010}, but it has
three advantages. First, the influence of the failure probability $\delta$ on the
constants is made explicit. Second, the factor $\rho(\bfSigma)$ appearing in the last
term is replaced by the smaller quantity $\mathscr G(\bfSigma^{1/2}\mathbb{B}_1^p)$.
Third, we improved the constants.

\Cref{pr:transfer:principle:1dim} is a useful technical tool that allows one to transfer the restricted eigenvalue property from the population covariance matrix to the empirical one. Following \cite{oliveira2013} we refer to \eqref{TP:1} as the transfer principle.
\end{remark}
\begin{proof}[Proof of  \Cref{pr:transfer:principle:1dim}]
Let $r>0$. We define define the sets
$$
V_{\bfSigma}(r):=\{\bv\in\re^p:\Vert\bfSigma^{1/2}\bv\Vert_2=1,\Vert\bv\Vert_1\le r\},
$$
and $\mathcal{B}:=\{\bfSigma^{1/2}\bv:\bv\in V_{\bfSigma}(r)\}$. Note that, if $\bxi\sim\mathcal{N}_p(0,\mathbf{I}_p)$,
\begin{align}
\mathscr{G}(\mathcal{B})
 &\le\esp\bigg[\sup_{\bv\in r\mathbb{B}^p_1} \bxi^\top\bfSigma^{1/2}\bv\bigg]
\le r\mathscr G(\bfSigma^{1/2}\mathbb B_1^p).\label{pr:transfer:principle:1dim:eq1}
\end{align}
Let $\widetilde \bfZ$ be a $n\times p$ matrix with iid $\mathcal{N}(0,1)$ entries such that $\bfZ=\widetilde\bfZ\bfSigma^{1/2}$. Clearly,
$$
\inf_{\bv\in V_{\bfSigma}(r)}\big\|\bfZ^{(n)}\bv\big\|_2=\inf_{\bb\in\mathcal{B}}\big\Vert\widetilde\bfZ^{(n)}\bb\big\Vert_2.
$$
The above equality, \eqref{pr:transfer:principle:1dim:eq1} and \Cref{lem:GW1} (noting that $\mathcal{B}\subset\mathbb{S}^{p-1}$) entails that, for all $r>0$ and $\delta\in(0,1]$, with probability at least $1-\delta$, the following inequality holds:
$$
\inf_{\bv\in V_{\bfSigma}(r)}\big\|\bfZ^{(n)}\bv\big\|_2\ge 1-\frac{1}{2n}-\sqrt{\frac{2\log(1/\delta)}{n}}-\frac{\mathscr G(\bfSigma^{1/2}\mathbb B_1^p)}{\sqrt{n}}r.
$$
We will now use the above property and \Cref{lemma:peeling:1dim} with constraint set $V:=\{\bv\in\re^p:\Vert\bfSigma^{1/2}\bv\Vert_2=1\}$,
$$
M(\bv):=\big\|\bfZ^{(n)}\bv\big\|_2-1+\frac{1}{2n},
$$
functions
$h(\bv):=\Vert\bv\Vert_1$,
$
g(r):=\frac{\mathscr G(\bfSigma^{1/2}\mathbb B_1^p)}{\sqrt{n}}r,
$
and constants $c:=1$ and $b:=\sqrt{2/n}$.
\Cref{lemma:peeling:1dim} implies that with probability at least $1-\delta$,
for all $\bv$ such that $\|\bfSigma^{1/2}\bv\|_2=1$, we have
\begin{align}
	M(\bv)=\big\|\bfZ^{(n)}\bv\big\|_2-1+\frac{1}{2n} \ge -1.2
	\frac{\mathscr G(\bfSigma^{1/2}\mathbb B_1^p)}{\sqrt{n}}\,\|\bv\|_1
	-\frac{3\sqrt{2}+\sqrt{2\log(9/\delta)}}{\sqrt{n}}.
\end{align}
Replacing $\bv$ by $\bu/\|\bf\Sigma^{1/2}\bu\|_2$, for an arbitrary
$\bu\in\mathbb R^p$, we get
\begin{align}
	\big\|\bfZ^{(n)}\bu\big\|_2\ge \bigg(1-\frac{1}{2n} -
	\frac{3\sqrt{2}+\sqrt{2\log(9/\delta)}}{\sqrt{n}}\bigg)\|\bfSigma^{1/2}\bu\|
	-1.2\frac{\mathscr G(\bfSigma^{1/2}\mathbb B_1^p)}{\sqrt{n}}\,\|\bu\|_1.
\end{align}
To complete the proof, it suffices to note that $(1/2\sqrt{n})+3\sqrt{2}\le 4.3$
for $n\ge 100$.
\end{proof}

\begin{proposition}\label{pr:gen:incoherence}
Let $\bfZ\in\re^{n\times p}$ be a random matrix with i.i.d.
$\mathcal{N}_p(0,\bfSigma)$ rows. For all $\delta\in(0,1]$
and $n\in\mathbb{N}$, with probability at least
$1-\delta$, the following property holds: for all
$[\bv;\bu]\in\re^{p+n}$,
\begin{align}
\left|\bu^\top\bfZ^{(n)}\bv\right|&\le \big\|{\bfSigma}^{1/2}\bv\big\|_2
\Vert\bu\Vert_2\sqrt{\frac{2}{n}}\left(4.8+\sqrt{\log(81/\delta)}\right)
\\
&\quad+1.2\Vert\bv\Vert_1\Vert\bu\Vert_2\frac{\mathscr
G(\bfSigma^{1/2}\mathbb B_1^p)}{\sqrt{n}}+1.2\big\|{\bfSigma}^{1/2}\bv\big\|_2
\frac{\mathscr G(\Vert\bu\Vert_1\mathbb B_1^n\cap \Vert\bu\Vert_2\mathbb B_2^n)}{\sqrt{n}}.
\end{align}
\end{proposition}
\begin{remark}
If, instead of \Cref{pr:gen:incoherence}, well-known upper bounds on the maximal singular
value of a Gaussian matrix, we get a sub-optimal result.  Indeed, upper tail bounds
on largest singular value imply that, with high-probability, for all $\bv$ and $\bu$,
$$
\left|\bu^\top\bfZ^{(n)}\bv\right|\le\Vert\bfSigma^{1/2}\bv\Vert_2\Vert\bu\Vert_2
\Vert\bfZ^{(n)}\bfSigma^{-1/2}\Vert_{op}\lesssim\Vert\bfSigma^{1/2}\bv\Vert_2
\Vert\bu\Vert_2\sqrt{\frac{p}{n}}.
$$
In case $\bv$ and $\bu$ are sparse, the previous lemma establishes a much sharp upper
bound with respect to dimension. One may see \Cref{pr:gen:incoherence} also as generalized
control on the ``incoherence'' between the column-space of $\bfZ^{(n)}$ and the identity
$\mathbf{I}_n$. This is particularly useful when the vectors are sparse as in our setting.
Alongside  \Cref{pr:transfer:principle:1dim}, \Cref{pr:gen:incoherence} is at the core of
our methodology to obtain improved near-optimal rates for corrupted sparse linear regression.
\end{remark}
\begin{proof}
Let $r_1,r_2>0$ and define the sets
\begin{align}
V_{\bfSigma,1}(r_1)&:=\{\bv\in\re^p:\Vert\bfSigma^{1/2}\bv\Vert_2=1,\Vert\bv\Vert_1\le r_1\},\\
V_{2}(r_2)&:=\{\bu\in\re^n:\Vert\bu\Vert_2=1,\Vert\bu\Vert_1\le r_2\}.
\end{align}
We also define the set $\mathcal{B}_1:=\{\bfSigma^{1/2}\bv:\bv\in V_{\bfSigma,1}(r_1)\}$.
By similar arguments used to establish \eqref{pr:transfer:principle:1dim:eq1}, we have the following
Gaussian width bounds:
\begin{align}
\mathscr{G}(\mathcal{B}_1)&\le r_1\mathscr G(\bfSigma^{1/2}\mathbb B_1^p),\quad
\mathscr{G}(V_{2}(r_2))\le r_2\mathscr{G}(\mathbb{B}_1^n\cap \mathbb{B}_2^n/r_2).
\label{pr:gen:incoherence:eq1}
\end{align}
Let $\widetilde \bfZ$ be a $n\times p$ matrix with iid $\mathcal{N}(0,1)$
entries such that $\bfZ=\widetilde\bfZ\bfSigma^{1/2}$. Clearly,
$$
\sup_{[\bv;\bu]\in V_{\bfSigma,1}(r_1)\times V_{2}(r_2)}
|\bu^\top\bfZ^{(n)}\bv|=\sup_{[\bv';\bu]\in
\mathcal{B}_1\times V_{2}(r_2)}|\bu^\top\widetilde\bfZ^{(n)}\bv'|.
$$
The above equality, \eqref{pr:gen:incoherence:eq1} and \Cref{lem:GW1}
(noting that $\mathcal{B}_1\subset\mathbb{S}^{p-1}$ and
$V_{2}(r_2)\subset\mathbb{S}^{n-1}$) entail that, for any $r_1,r_2>0$
and $\delta\in(0,1]$, with probability at least $1-\delta$, the following inequality holds:
\begin{align}
\sup_{[\bv;\bu]\in V_{\bfSigma,1}(r_1)\times V_{2}(r_2)}|\bu^\top\bfZ^{(n)}\bv|
\le  \frac{\mathscr G(\bfSigma^{1/2}\mathbb B_1^p)}{\sqrt{n}}r_1 +
\frac{\mathscr G(\mathbb B_1^n\cap \mathbb{B}_2^n/r_2)}{\sqrt{n}} r_2+\sqrt{\frac{2\log(1/\delta)}{n}}.
\end{align}
We use the above property and \Cref{lemma:peeling:2dim} with constraint sets
$V_1:=\{\bv\in\re^p:\Vert\bfSigma^{1/2}\bv\Vert_2=1\}$ and
$V_2:=\{\bu\in\re^n:\Vert\bv\Vert_2=1\}$, functions
$M(\bu):=|\bu^\top\bfZ^{(n)}\bv|$ and
\begin{align}
h(\bv):=\Vert\bv\Vert_1,\quad\quad\bar h(\bu):=\Vert\bu\Vert_1,\qquad
g(r_1):=\frac{\mathscr G(\bfSigma^{1/2}\mathbb B_1^p)}{\sqrt{n}}\,r_1,
\quad\bar g(r_2):=\frac{\mathscr G(\mathbb B_1^n\cap \mathbb{B}_2^n/r_2)}{\sqrt{n}}\,r_2,
\end{align}
and constants $c:=1$ and $b:=\sqrt{2/n}$.
The desired inequality follows from \Cref{lemma:peeling:2dim} combined with the fact that
\begin{align}
\left[\frac{\bv}{\Vert\bfSigma^{1/2}\bv\Vert_2};\frac{\bu}{\Vert\bu\Vert_2}\right]\in
V_{\bfSigma,1}(r_1)\times V_{2}(r_2),
\end{align}
for all $[\bv;\bu]\in\re^p\times\re^n$ and the homogeneity of norms.
\end{proof}

\begin{lemma}[$\TP+\IP\Rightarrow\ATP$]\label{lemma:aug:transfer:principle}
Let $\bfZ\in\re^{n\times p}$ be a matrix satisfying $\TP(\sa_1;\sa_2)$
and $\IP(\sb_1;\sb_2;\sb_3)$ for some positive numbers $\sa_1$, $\sa_2$,
$\sb_1$, $\sb_2$ and $\sb_3$.  Then, for any $\alpha>0$,
$\bfZ$ satisfies the $\ATP(\sc_1;\sc_2;\sc_3)$ with constants
$\sc_1=\sqrt{\sa_1^2- \sb_1 -\alpha^2}$, $\sc_2=\sa_2+\sb_2/\alpha$
and $\sc_3=\sb_3/\alpha$. Taking $\alpha = \sa_1/2$, we obtain that
$\ATP(\sc_1;\sc_2;\sc_3)$ holds with constants
$\sc_1=\sqrt{(3/4)\sa_1^2- \sb_1 -\alpha^2}$, $\sc_2=\sa_2+2\sb_2/\sa_1$
and $\sc_3=2\sb_3/\sa_1$.
\end{lemma}
\begin{proof}
Simple algebra and the TP property entail
\begin{align}
\sc_1\Big\{\Vert\bfSigma^{1/2}\bv\Vert_2^2+\Vert\bu\Vert_2^2\Big\}^{1/2}&
=\Big\{\sa_1^2\Vert\bfSigma^{1/2}\bv\Vert_2^2+\sa_1^2\Vert\bu\Vert_2^2-
(\sb_1+\alpha^2)(\Vert\bfSigma^{1/2}\bv\Vert_2^2+\Vert\bu\Vert_2^2)\Big\}^{1/2}\\
&\stackrel{\TP}{\le}
\Big\{\big(\Vert\bfZ^{(n)}\bv\Vert_2+\sa_2\|\bv\|_1\big)^2+\sa_1^2\Vert\bu\Vert_2^2-
(\sb_1+\alpha^2)(\Vert\bfSigma^{1/2}\bv\Vert_2^2+\Vert\bu\Vert_2^2)\Big\}^{1/2}\\
&\le \Big\{\Vert\bfZ^{(n)}\bv\Vert_2^2+\Vert\bu\Vert_2^2-
(\sb_1+\alpha^2)(\Vert\bfSigma^{1/2}\bv\Vert_2^2+\Vert\bu\Vert_2^2)\Big\}^{1/2}+\sa_2\Vert\bv\Vert_1.
\end{align}
By Young's inequality and IP, we get
\begin{align}
\Vert\bfZ^{(n)}\bv\Vert_2^2+\Vert\bu\Vert_2^2
&= \Vert\bfZ^{(n)}\bv+\bu\Vert_2^2 - 2\bu^\top\bfZ^{(n)}\bv\\
&\stackrel{\IP}{\le} \Vert\bfZ^{(n)}\bv + \bu\Vert_2^2+
 2\sb_1\big\|{\bfSigma}^{1/2}\bv\big\|_2\Vert\bu\Vert_2
+2\sb_2\Vert\bv\Vert_1\Vert\bu\Vert_2+2\sb_3\big\|{\bfSigma}^{1/2}
\bv\big\|_2\Vert\bu\Vert_1\\
&\stackrel{\rm Young}{\le\ } \Vert\bfZ^{(n)}\bv + \bu\Vert_2^2 +
(\sb_1+\alpha^2)\Big(\Vert\bfSigma^{1/2}\bv\Vert_2^2+\Vert\bu\Vert_2^2\Big)
+\frac{\sb_2^2}{\alpha^2}\Vert\bv\Vert_1^2
+\frac{\sb_3^2}{\alpha^2}\Vert\bu\Vert_1^2.
\end{align}
To get the claimed result, it suffices to put the previous two inequalities
together and to rearrange the terms.
\end{proof}

\Cref{pr:transfer:principle:1dim}, \Cref{pr:gen:incoherence} and
\Cref{lemma:aug:transfer:principle} entail immediately that the $\ATP$
holds with high-probability.

\begin{corollary}[$\ATP$ property for correlated Gaussian designs]
\label{cor:aug:transfer:principle}
Let $\bfZ\in\re^{n\times p}$ be a random matrix with iid $\mathcal{N}_p(0,\bf\Sigma)$
rows. Suppose  $\delta\in(0,1/7]$, $n\ge 100$ and $\alpha>0$ are such that
\begin{align}
    {C}_{n,\delta}:=\bigg(1-\frac{4.3+\sqrt{2\log(9/\delta)}}{\sqrt{n}}\bigg)^{2}
		-\sqrt{\frac{2}{n}}\big(4.8+\sqrt{\log(81/\delta)}\big)-\alpha^2>0.
		\label{equation:constant:Cn}
\end{align}
Then, with probability at least  $1-2\delta$, the following property holds:
for all $[\bv;\bu]\in\re^{p+n}$,
\begin{align}
\Vert\bfZ^{(n)}\bv + \bu\Vert_2
		\ge {C}_{n,\delta}^{1/2}\left\Vert[\bfSigma^{1/2}\bv;\bu]
				\right\Vert_2-1.2\left(1+\frac{1}{\alpha}\right)
				\frac{\mathscr{G}(\bfSigma^{1/2}\mathbb{B}_1^p)}{\sqrt{n}}
				\Vert\bv\Vert_1-\frac{1.2}{\alpha}
				\frac{\mathscr{G}(\Vert\bu\Vert_1\mathbb{B}_1^n\cap
                \Vert\bu\Vert_2\mathbb{B}_2^n)}{\sqrt{n}}.
\end{align}
\end{corollary}

\begin{remark}\label{rem:3}
The particular choice $\alpha= 1/2$, in conjunction with
the bound \eqref{calG} on the Gaussian width,  leads to the simpler bound
\begin{align}
\Vert\bfZ^{(n)}\bv + \bu\Vert_2\ge{C}_{n,\delta}^{1/2}
\left\Vert[\bfSigma^{1/2}\bv;\bu]\right\Vert_2-
\frac{3.6\mathscr{G}(\bfSigma^{1/2}\mathbb{B}_1^p)}
{\sqrt{n}}\Vert\bv\Vert_1-2.4\sqrt{\frac{2\log n}{n}}\Vert\bu\Vert_1
\end{align}
with
$$
C_{n,\delta} = \frac34-\frac{17.5+9.6\sqrt{2\log(2/\delta)}}{\sqrt{n}}
$$
\end{remark}

\begin{remark}\label{rem:4}
If the goal was to fight against logarithmic factors, we could use a
tighter bound on the Gaussian width of a convex polytope
\citep[Prop.~1]{bellec2019}. It allows us to replace the term
$\sqrt{2\log n}\,\|\bu\|_1$ by $4\sqrt{1\vee \log (8en\|\bu\|_2^2/\|\bu\|_1^2 )}\,
\|\bu\|_1$. On the one hand, if $\|\bu\|_1^2\ge (o/e)\|\bu\|_2^2$, then
\begin{align}\label{d1}
    4\sqrt{1\vee \log (8en\|\bu\|_2^2/\|\bu\|_1^2 )}\,\|\bu\|_1 \le
        4\sqrt{1\vee \log (8e^2n/o)}\,\|\bu\|_1.
\end{align}
On the other hand, if $\|\bu\|_1^2\le o\|\bu\|_2^2$, then we can use the fact that
the function $x\mapsto x\sqrt{1\vee \log(e/x^2)}=:\varphi(x)$ is  increasing, we get
\begin{align}
    4\sqrt{1\vee \log (8en\|\bu\|_2^2/\|\bu\|_1^2 )}\,\|\bu\|_1
    &= 4\sqrt{8en}\,\|\bu\|_2 \varphi\Big(\frac{\|\bu\|_1}{\sqrt{8n}\,\|\bu\|_2 }\Big)\\
    &\le 4\sqrt{8en}\,\|\bu\|_2 \varphi\big(\sqrt{o/8en}\big)\\
    & = 4\sqrt{e o}\,\|\bu\|_2 \sqrt{1+ \log(8n/o)}.\label{d2}
\end{align}
Combining \eqref{d1} and \eqref{d2}, we get
\begin{align}
    \mathscr{G}(\Vert\bu\Vert_1\mathbb{B}_1^n\cap\Vert\bu\Vert_2\mathbb{B}_2^n)
    \le 4(\|\bu\|_1+\sqrt{o}\,\|\bu\|_2) \sqrt{2+ \log(8n/o)}.
\end{align}
If the proportion $o/n$ is fixed, or tends to zero at a rate slower than polynomial
in $n$, this latter bound can be used to remove logarithmic terms.
\end{remark}

\section{Propositions imply theorems}

The three theorems stated in the main body of the paper are simple
consequences of the propositions established in this supplementary
material. The aim of this section is to quickly show how the theorems
can be derived from the corresponding propositions. 

\paragraph{Proof of \Cref{th:1}}

\Cref{th:1} is essentially a simplified version of \Cref{prop:improved:rate}. 
First, note that condition on $\lambda$ in \Cref{th:1}, combined with 
the well-known upper bounds on the tails of maxima of Gaussian random variables 
\citep{boucheron:lugosi:massart2013}, implies that $\lambda$ satisfies
condition (iv) of  \Cref{prop:improved:rate}. Furthermore, under the conditions
of the theorem, conditions (i)-(iii) of \Cref{prop:improved:rate}, as well
as \eqref{cond1} and \eqref{cond2}, are satisfied with $\gamma = 1$, 
$\sa_1 = \sc_1\le 1$, $\sa_2 = \sc_2$ and $\sb_1 =0$. Replacing all these
values in the inequality of \Cref{prop:improved:rate}, we get the claim
of \Cref{th:1}.

\paragraph{Proof of \Cref{th:2}}
From \Cref{pr:transfer:principle:1dim} and the fact that $\mathscr G
(\bfSigma^{1/2}\mathbb B_1^p) \le \sqrt{2\log p}$, we infer that 
the $\TP$ is satisfied with appropriate constants $\sa_1,\sa_2$ with 
probability at least $1-\delta$. Similarly, \Cref{pr:gen:incoherence} 
and the aforementioned bound on the Gaussian width imply that 
the $\IP$ is satisfied with appropriate constants with probability
at least $1-\delta$. In the intersection of these two events, according
to \Cref{rem:3}, $\ATP$ is satisfied with $\sc_1$, $\sc_2$ and $\sc_3$
as in the claim of \Cref{th:2}.

\paragraph{Proof of \Cref{th:3}}
Under the condition $\delta\ge 2e^{-\sd_2 n}$, we check that $\sa_1$ and
$\sc_1$ are constants. Therefore, combining the claims of \Cref{th:1}, 
\Cref{th:2}  and \Cref{lemma:penalization:factors}, we get the claim
of \Cref{th:3}. 

\end{document}

%% file: drawing-1b.pdf_tex

\begingroup
  \makeatletter
  \providecommand\color[2][]{%
    \errmessage{(Inkscape) Color is used for the text in Inkscape, but the package 'color.sty' is not loaded}
    \renewcommand\color[2][]{}%
  }
  \providecommand\transparent[1]{%
    \errmessage{(Inkscape) Transparency is used (non-zero) for the text in Inkscape, but the package 'transparent.sty' is not loaded}
    \renewcommand\transparent[1]{}%
  }
  \providecommand\rotatebox[2]{#2}
  \ifx\svgwidth\undefined
    \setlength{\unitlength}{841.4125pt}
  \else
    \setlength{\unitlength}{\svgwidth}
  \fi
  \global\let\svgwidth\undefined
  \makeatother
  \begin{picture}(1,0.3818717)%
    \put(0,0){\includegraphics[width=\unitlength]{drawing-1b.pdf}}%
    \put(0.66296552,0.32376733){\color[rgb]{0,0,0}\makebox(0,0)[lb]{\smash{\Cref{prop:augmented:analysis:v2}}}}%
    \put(0.63500688,0.29524386){\color[rgb]{0,0,0}\makebox(0,0)[lb]{\smash{(sub-optimal rate)}}}%
    \put(0.66296552,0.13361089){\color[rgb]{0,0,0}\makebox(0,0)[lb]{\smash{\Cref{prop:improved:rate}}}}%
    \put(0.62568733,0.10508743){\color[rgb]{0,0,0}\makebox(0,0)[lb]{\smash{(nearly optimal rate)}}}%
    \put(0.81207829,0.22868911){\color[rgb]{0,0,0}\makebox(0,0)[lb]{\smash{\Cref{lemma:recursion:delta:bb}}}}%
    \put(0.47657456,0.29524386){\color[rgb]{0,0,0}\makebox(0,0)[lb]{\smash{ATP}}}%
    \put(0.40201818,0.23819693){\color[rgb]{0,0,0}\makebox(0,0)[lb]{\smash{TP}}}%
    \put(0.40201818,0.18115){\color[rgb]{0,0,0}\makebox(0,0)[lb]{\smash{IP}}}%
    \put(0.22494676,0.30475168){\color[rgb]{0,0,0}\makebox(0,0)[lb]{\smash{(1p peeling)}}}%
    \put(0.22494676,0.08607178){\color[rgb]{0,0,0}\makebox(0,0)[lb]{\smash{(2p peeling)}}}%
    \put(0.23426631,0.11459525){\color[rgb]{0,0,0}\makebox(0,0)[lb]{\smash{\Cref{lemma:peeling:2dim}}}}%
    \put(0.23426631,0.33327515){\color[rgb]{0,0,0}\makebox(0,0)[lb]{\smash{\Cref{lemma:peeling:1dim}}}}%
    \put(0.05719489,0.24770475){\color[rgb]{0,0,0}\makebox(0,0)[lb]{\smash{\Cref{lemma:aux1}}}}%
    \put(0.0571949,0.19541173){\color[rgb]{0,0,0}\makebox(0,0)[lb]{\smash{\Cref{lem:GW1}}}}%
    \put(0.03389603,0.17164218){\color[rgb]{0,0,0}\makebox(0,0)[lb]{\smash{(Chevet ineq.)}}}%
    \put(0.01059716,0.06705614){\color[rgb]{0,0,0}\makebox(0,0)[lb]{\smash{Gaussian design}}}%
    \put(0.77141118,0.06705614){\color[rgb]{0,0,0}\makebox(0,0)[lb]{\smash{General design}}}%
    \put(0.79993463,0.20016564){\color[rgb]{0,0,0}\makebox(0,0)[lb]{\smash{(KKT for $\bbeta$)}}}%
    \put(0.30552791,0.25721257){\color[rgb]{0,0,0}\makebox(0,0)[lb]{\smash{Prop \ref{pr:transfer:principle:1dim}}}}%
    \put(0.30552791,0.16213436){\color[rgb]{0,0,0}\makebox(0,0)[lb]{\smash{Prop \ref{pr:gen:incoherence}}}}%
    \put(0.50519216,0.24770475){\color[rgb]{0,0,0}\makebox(0,0)[lb]{\smash{\Cref{lemma:aug:transfer:principle}}}}%
    \put(0.0393089,0.00050139){\color[rgb]{0,0,0}\makebox(0,0)[lb]{\smash{\Cref{ss9.1}}}}%
    \put(0.22946533,0.00050139){\color[rgb]{0,0,0}\makebox(0,0)[lb]{\smash{\Cref{ss:peeling}}}}%
    \put(0.41011395,0.00050139){\color[rgb]{0,0,0}\makebox(0,0)[lb]{\smash{\Cref{ss9.3}}}}%
    \put(0.71436424,0.00050139){\color[rgb]{0,0,0}\makebox(0,0)[lb]{\smash{\Cref{sec:pr1}}}}%
  \end{picture}%
\endgroup